\documentclass[a4paper,11pt,times,authoryears]{amsart}
\usepackage{geometry}
\usepackage{bm}
\usepackage{amsmath, amssymb}
\usepackage{graphicx}
\usepackage{mathrsfs}
\usepackage[mathscr]{eucal}
\usepackage{booktabs}
\usepackage{multicol,multirow}
\usepackage{epstopdf}
\usepackage{enumerate}
\usepackage[colorlinks,citecolor=blue]{hyperref}
\usepackage{adjustbox}
\usepackage{caption}
\usepackage{subcaption}
\usepackage{enumitem, array}
\usepackage{algorithm}
\usepackage{algorithmic}
\usepackage{float}
\usepackage{tikz}
\usetikzlibrary{decorations.pathreplacing,calc}

\newcommand{\tikzmark}[2][-3pt]{\tikz[remember picture, overlay, baseline=-0.5ex]\node[#1](#2){};}

\tikzset{brace/.style={decorate, decoration={brace}},
 brace mirrored/.style={decorate, decoration={brace,mirror}},
}

\newcounter{brace}
\newcounter{arrow}
\newcommand{\drawcurvedarrow}[3][]{%
 \refstepcounter{arrow}
 \tikz[remember picture, overlay]\draw (#2.center)edge[#1]node[coordinate,pos=0.5, name=arrow-\thearrow]{}(#3.center);
}

\numberwithin{equation}{section}

\def\bzero{{\mathbf 0}}

 \def\bC{{\mathbf C}}  
 \def\bS{{\mathbf S}} 
   \def\bI{{\mathbf I}}
   
\def\bR{{\mathbf R}} \def\bS{{\mathbf S}}  
  \def\bX{{\mathbb X}} \def\bY{{\mathbb Y}}

\def\bM{\bm M}

\def\br{{\mathbf r}}  
  
\def\bx{{\mathbf x}} \def\by{{\mathbf y}} 

\def\bbeta{{\bm{\beta}}}

\def\bSigma{{\bm{\Sigma}}}

\def\hbbeta{\widehat{\bm \beta}}

\def\argmin{\mathop{\rm arg\, min}}

\def\diag{\hbox{diag}}
    
 \def\hbeta{\widehat{\beta}}

\def\real{\mathop{{\rm I}\kern-.2em\hbox{\rm R}}\nolimits}

\def\1overn{\frac{1}{n}}

\def\bel{\begin{eqnarray}\label}  \def\eel{\end{eqnarray}}
\def\bes{\begin{eqnarray*}}  \def\ees{\end{eqnarray*}}

\newtheorem{theorem}{Theorem}[section]

\theoremstyle{definition}

\theoremstyle{remark}
\newtheorem{remark}[theorem]{Remark}

\begin{document}

\title[Restricted Bridge Estimation]{Penalized regression via the
restricted bridge estimator 
%($L_q$-type Penalty Estimation Under the Linear Restriction)
}

%\author[short version for running head]{name for top GRR paper}

\author{Bahad{\i}r Y\"{u}zba\c{s}{\i}$^\dag$, Mohammad Arashi$^\ddag$ \and Fikri Akdeniz$^\S$}

\date{\today}
\maketitle

{\footnotesize
\center { \text{  $^\dag$Department of Econometrics}\par
  { \text{ Inonu University, Turkey}}\par
  { \texttt{E-mail address:b.yzb@hotmail.com}}
  \par

  \vskip 0.2 cm
  
  \text{  $^\ddag$Department of Statistics, Faculty of Mathematical Sciences}\par{\text{ Shahrood University of Technology, Iran}}\par
    { \texttt{E-mail address:m\_arashi\_stat@yahoo.com}}

  \vskip 0.2 cm
  
  \text{  $^\S$Department of Mathematics and Computer Science}\par
  { \text{Cag University, Turkey}}\par
  { \texttt{E-mail address :fikriakdeniz@gmail.com }}

}}

\renewcommand{\thefootnote}{}

\footnote { \par

 Corresponding author : Bahad{\i}r Y\"{u}zba\c{s}{\i}

}

\date{\today}

\begin{abstract}
This article is concerned with the Bridge Regression, which is a special family in penalized regression with penalty function $\sum_{j=1}^{p}|\beta_j|^q$ with $q>0$, in a linear model with linear restrictions. The proposed restricted bridge (RBRIDGE) estimator simultaneously estimates parameters and selects important variables when a prior information about parameters are available in either low dimensional or high dimensional case. Using local quadratic approximation, the penalty term can be approximated around a local initial values vector and the RBRIDGE estimator enjoys a closed-form expression which can be solved when $q>0$. Special cases of our proposal are the restricted LASSO ($q=1$), restricted RIDGE ($q=2$), and restricted Elastic Net ($1< q < 2$) estimators. We provide some theoretical properties of the RBRIDGE estimator under for the low dimensional case, whereas the computational aspects are given for both low and high dimensional cases.  An extensive Monte Carlo simulation study is conducted based on different prior pieces of information and the performance of the RBRIDGE estiamtor is compared with some competitive penalty estimators as well as the ORACLE. We also consider four real data examples analysis for comparison sake. The numerical results show that the suggested RBRIDGE estimator outperforms outstandingly when the prior is true or near exact. %as well as some open problems.

\end{abstract}

\maketitle

\section{Introduction}
\noindent
Under a linear regression setup, assume $\bbeta=(\beta_1,\ldots,\beta_p)^\top$ is the vector of regression coefficients. Further, assume $\bbeta$ is subjected to lie in a sub-space restriction with form 
\begin{equation}\label{eq-restriction}
\bR\bm\beta=\br,
\end{equation}
where $\bR$ is an $m\times p$ ($m<p$) matrix of constants and $\br$ is an $m$-vector of known prespecified constants. This restriction may be 
\begin{description}
    \item[(a)] a fact is known from theoretical or experimental considerations
    \item[(b)] a
    the hypothesis that may have to be tested, or
    \item[(c)] an artificially imposed condition to reduce or
    eliminate redundancy in the description of the model.
\end{description}
From a practical viewpoint, \cite{don1982restrictions} explained how accounting identities burden some exact restrictions on the endogenous variables in econometric models. \cite{xu2012stein} motivated the problem of estimating angles subject to summation. Recently \cite{kleyn2017preliminary} used a priori restriction present in labor and capital input in the estimation of Cobb-Douglass production function and inferred about the economic model using a preliminary testing approach.  

Estimation with restriction \eqref{eq-restriction} has been considered by many to decrease the mean squared error (MSE) of estimation and mean prediction error (MPE) in regression modeling. 
\cite{roozbeh2015shrinkage} extended restricted ridge and the follow-up shrinkage strategies for the partially linear models and \cite{roozbeh2016robust} considered the robust extension of the latter work in restricted partially linear models. 
\cite{tuacc2017variable} proposed a restricted LASSO in the restricted regression model, while \cite{norouzirad2018preliminary} developed shrinkage estimators using LASSO and proposed a restricted LASSO to decrease the mean prediction error of estimation compared to the LASSO of \cite{tibshirani1996regression}. In a recent study, \cite{saleh2018rank} compared restricted estimators with LASSO and ridge in rank regression. For an extensive overview of restricted estimation in regression modeling and related shrinkage techniques, we refer to \cite{rao2003linear}, \cite{saleh2006theory}, and \cite{radhakrishna2008linear}.

Our purpose here is sparse estimation with improving prediction accuracy, utilizing applying regularization techniques in regression modeling. However, the problem under study is different from the generalized LASSO (GLASSO) of \cite{tibshirani2011solution} and hence does not involve a penalty matrix. Simply we define restricted penalized estimator by imposing the sub-space restriction $\bR\bm\beta=\br$ to the estimation of the true parameter. We do not need any specific regularity assumption for the uniqueness as described by \cite{ali2019generalized} as in the GLASSO. To be more specific, we use the solution of the bridge regularization technique as the ``base estimator" and couple it with the specified restriction to obtain a closed-form restricted bridge estimator. As an instance, the restricted LASSO is a closed-form estimator based on the LASSO. Hence, our contribution has the following highlights:
\begin{description}
    \item[] Comparing to the existing methods (e.g., the GLASSO), the computational and temporal costs of our method are negligible.
    \item[] It can be easily extended to other regularization techniques for which the local quadratic approximation (LQA) of \cite{fan2001variable} can be applied for the penalty function, such as SCAD.
    \item[] It improves the prediction accuracy of the base estimator and hence decrease the MPE.
    \item[] It is consistent in estimation under the same regularity conditions as in the base estimator and is unique.
\end{description}
By the above description, the plan of this article is as follows. In Section~\ref{sec:SM}, we describe the linear model along with bridge regularization technique. The restricted bridge is defined in Section~\ref{sec:RB}, where we also derive its MSE and prove its consistency. Section~\ref{sec:sim} is devoted to two Monte Carlo simulation examples and the analysis of four real data examples are are given in Section~\ref{sec:RDA} for performance analysis of the proposed restricted bridge estimator. We conclude our study in Section~\ref{sec:conc}. Proofs of technical statements are given in the Appendix, to better focus on the computational part in the body of paper.

\section{Statistical Model}
\label{sec:SM}
\noindent
Consider the following linear model
\begin{equation}\label{full:model}
\bY=\bX_n \bm\beta+ \bm \varepsilon, 
\end{equation}
where $\bY=(Y_1,\ldots,Y_n)^\top$ is the vector of response variables, $\bX_n=(\bx_1,\ldots,\bx_n)^\top$ is the non-stochastic design matrix including $p$-dimensional covariates $\bx_i\in\mathbb{R}^p$, and $\bm\varepsilon=(\varepsilon_1,\ldots,\varepsilon_n)^\top$ is the error term with $\mathbb{E}(\bm\varepsilon)=\boldsymbol{0}$ and $\mathbb{E}(\bm\varepsilon\bm\varepsilon^\top)=\boldsymbol{I}_n$. We also assume the observations $\{(\bx_i,Y_i)\}_{i=1}^n$ are centered, so there is no intercept in the model and the focus of this study is the estimation of $\bm\beta$. 

The least squares (LS) estimator of $\bbeta$ is given by $\tilde{\bbeta}=\bC_n^{-1}\bX_n^\top\bY$, with $\bC_n = \bX_n^\top\bX_n$. Under LS theory, as $n\to\infty$, we get $\sqrt{n}(\tilde{\bbeta}-\bbeta)\overset{\mathcal{D}}{\to}N_p(\boldsymbol 0,\bC^{-1})$, where we assumed the following regularity conditions hold
\begin{enumerate}
    \item [(A1)] $\max_{1\leq i\leq n}\bx_i^\top \bC_n^{-1}\bx_i\to0$ as $n\to\infty$, where $\bx_i^\top$ is the ith row of $\bX_n$;
    \item [(A2)] $n^{-1}\bC_n\to\bC$, where $\bC$ is a finite positive definite (p.d.) matrix.
\end{enumerate}
Despite simplicity of the LS estimator and that is the best linear unbiased estimator, its efficiency diminishes, in the MSE sense, in multicollinear and sparse situations. In these scenarios, regularization techniques are used to penalize large values of true regression parameters. As such we refer to the ridge and LASSO estimation methods, where $\textnormal{L}_2$ and $\textnormal{L}_1$ norms are used for the penalty term, respectively. \cite{frank1993statistical} introduced a class of regularization techniques called bridge, for which the penalty term allowed to very on $(0,\infty)$, i.e., $\textnormal{L}_q$ norm with $q>0$. This class includes the ridge and LASSO as special members. The bridge estimator can be obtained by solving the following optimization problem
\begin{equation*}
    \min_{\bbeta\in\mathbb R^p}(\bY-\bX_n\bm\beta)^\top(\bY-\bX_n\bm\beta)\quad\textnormal{subject to}\quad\boldsymbol{1}_p^\top|\bm\beta|^q \leq t,\quad t\geq0,
\end{equation*}
where $\boldsymbol{1}_p$ is a $p$-tuple of 1's and $|\bm\beta|=(|\beta_1|,\ldots,|\beta_p|)^\top$. Consequently the bridge estimator is obtained by solving the dual form as
\begin{equation}
\hbbeta_n=\argmin_{\bbeta}\left\{\sum_{i=1}^{n}(Y_i-\bm x_i^\top\bm\beta)^2+\lambda\sum_{j=1}^{p}|\beta_j|^q\right\},
\end{equation}
where $\lambda$ is the tuning parameter and $q>0$.

\begin{figure}[!htbp]
\centering
\includegraphics[width=16cm,height=10cm]{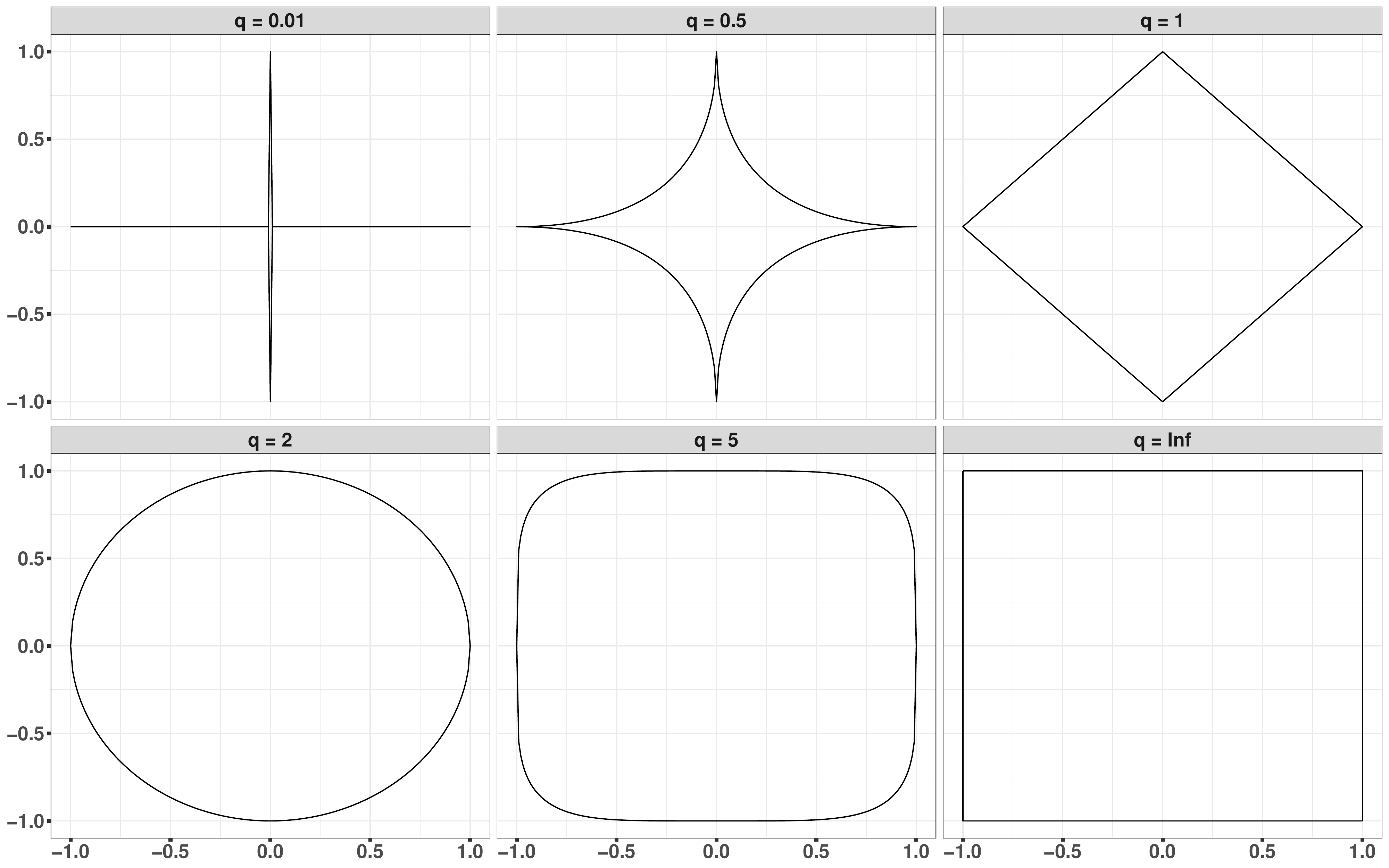}
\caption{Constrained area of the bridge estimator with $t=1$
 \label{Fig:2d_pen}}
\end{figure}

\begin{figure}[!htbp]
\centering
\includegraphics[width=16cm,height=10cm]{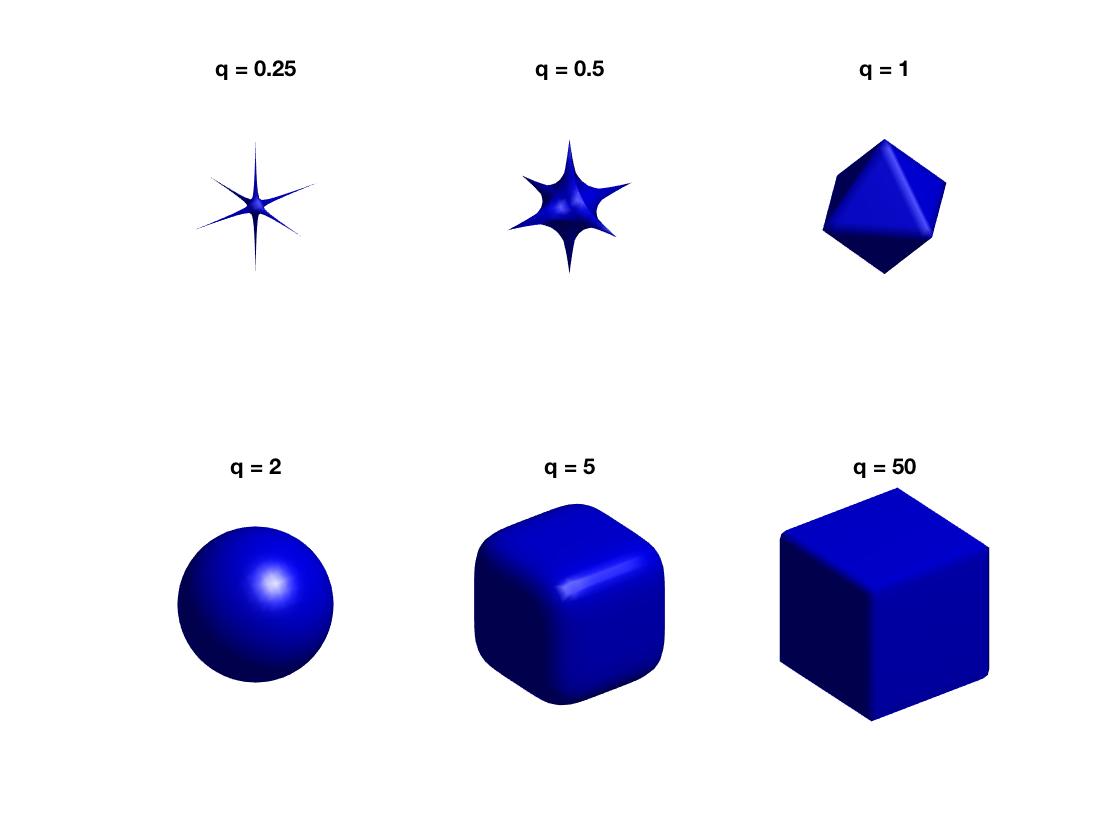}
\caption{Shapes of different norms in 3D
 \label{Fig:3d_pen}}
\end{figure}

Figure \ref{Fig:2d_pen} shows the constrained area of the bridge estimator with $t=1$. Apparently, as $q$ increases the constrained area of penalty term widen, while for $q\to0$, there is no penalty area and we get the LS solution, as we expect. 

Using the local quadratic approximation (LQA) of \cite{fan2001variable}, the penalty term can be approximated around a local vector $\bm\beta^o$ and the bridge estimator has the following closed form
\begin{equation}\label{eq-2.3}
\hbbeta_n=\left[\bC_n+\bSigma_\lambda(\hbbeta^o)\right]^{-1}\bX_n^\top\bY,
\end{equation} 
where $\bSigma_{\lambda}(\hbbeta^o) =\frac{\lambda q}{2} \diag\left( \left| \hat\beta_1^{o} \right|^{q-2},\dots,\left| \hat\beta_p^{o} \right|^{q-2}  \right) $. See \cite{park2011bridge}, for details.

\section{Restricted Estimator}
\label{sec:RB}
In order to couple the bridge estimator (termed as ``base estimator" in the Introduction) with the specified restriction \eqref{eq-restriction}, we solve the following optimization problem
\begin{eqnarray*}
&&\min_{\bbeta\in\mathbb R^p} (\bY-\bX_n\bm\beta)^\top(\bY-\bX_n\bm\beta)\cr
\textnormal{s.t.}&& \bbeta^\top\diag\left( \left| \hat\beta_1^{o} \right|^{q-2},\dots,\left| \hat\beta_p^{o} \right|^{q-2}  \right)\bbeta^\top\leq t,\;t>0\quad\mbox{and} \quad\bR\bm\beta=\br,
\end{eqnarray*}
where $(\beta^o_1,\ldots,\beta^0_p)$ are local points used in the LQA of \cite{fan2001variable}. The solution of the above problem reveals the restricted bridge (BRIDGE) estimator in a closed form, as stated in the following result. 
\begin{theorem}\label{theorem:restricted bridge}
Assume linear model \eqref{full:model}.    Under the sub-space restriction \eqref{eq-restriction}, the RBRIDGE estimator is given by
\begin{equation}\label{restricted bridge}
\hbbeta^{\rm R}_n=\hbbeta_n-(\bC+\bSigma_{\lambda}(\hbbeta^o))^{-1}\bR^\top\left[ \bR (\bC+\bSigma_{\lambda}(\hbbeta^o))^{-1}\bR^\top\right]^{-1} \left( \bR\hbbeta_n-\br \right). 
\end{equation}
\end{theorem}

\begin{theorem}\label{theorem:MSE}
    Under the assumptions of Theorem \ref{theorem:restricted bridge} and linear restriction \eqref{eq-restriction}, we have
\begin{equation*}
\mathrm{MSE}(\hat\bbeta_n^R)=\mathrm{tr}(\bM[\bS+\bSigma_\lambda(\hbbeta^o)\bbeta^\top\bbeta\bSigma_\lambda(\hbbeta^o)]\bM)
\end{equation*}
where 
$\bS=\bC_n+\bSigma_\lambda(\hbbeta^o)$ and $\bM=\bS^{-1}-\bS^{-1}\bR^\top(\bR\bS^{-1}\bR^\top)^{-1}\bR\bS^{-1}$.
\end{theorem}
In the following result, we give the consistency property of the RBRIDGE estimator. For our purpose we assume $\lambda$ is dependent to $n$ and let $\lambda=\lambda_n$. 
\begin{theorem}\label{theorem:Consistency}
    Under the assumptions of Theorem \ref{theorem:MSE}, the RBRIDGE estimator, with $\bS$ replaced by $\bS_n$, is consistent in estimation of $\bbeta$ if $\lambda_n=o(n)$, where 
    \begin{equation*}
    \bS_n=\frac1n\bC_n+\frac1n\bSigma_{\lambda_n}(\hbbeta^o).
    \end{equation*}
\end{theorem}
Since the RBRIDGE estimator has closed form, it is relatively simple to be computed. However, we use the LQA method in the body of optimization problem and hence we need to explain about the computation of RBRIDGE estimator. This task is taken care in the next section.   
\subsection{Computation of RBRIDGE estimator}
Here, we briefly will outline the building block of our algorithm for computing the RBRIDGE estimator $\hbbeta_n^{\rm R}$. 

With the aid of LQA, we derived a closed-form restricted bridge estimator (see Theorem \ref{theorem:restricted bridge}). For a fast and high-performance computational algorithm, we specifically use the RcppArmadillo language of \cite{eddelbuettel2014rcpparmadillo}. For the selection of the tuning parameter $\lambda$, we use the following Cross-Validation (CV) method: 
\begin{itemize}
    \item Divide the data into $K$ roughly equal parts.
    \item For each $k=1,2,\dots,K$, compute the target estimator $\hbbeta_n^{-k}$, say using the rest of $K-1$ parts of the data.
    \item Compute the mean prediction error for the $k^{th}$ cycle by
    $$\mathrm{PE}_k(\lambda) = \frac{1}{n}\sum_{i\in k^{th} \mathrm{part}}(Y_i-\bm x_i^\top\hbbeta_n^{-k})^2$$.
    \item Finally, for many values of $\lambda$ calculate $\mathrm{CV_{\lambda}} = \frac{1}{K}\sum_{i=1}^K{\mathrm{PE}_k(\lambda)}$ and choose the $\lambda$ for which the smallest $\mathrm{CV_{\lambda}}$ is achieved.
\end{itemize}
In our computation, since the objective function (see the proof of Theorem \ref{theorem:restricted bridge}) is convex, $\hbbeta_n^{\rm R}$ satisfies the necessary and sufficient Karush-Kuhn-Tucker (KKT) conditions, and hence we do not necessarily need them to be discussed here. Algorithm 1 below clarifies the steps for the whole computation procedure.
\begin{algorithm}[h]
    \caption{Details of the RBRIDGE algorithm}
    \label{algorithm}
    \begin{algorithmic}
    \item[Step 1.] Fix the values $q$ and $\lambda$.
    \item[Step 2.] Set the initial value $\hbbeta^o$. 
    \item[Step 3.] Set the matrices $\bR$ and $\br$ and compute
    \begin{eqnarray*}
    (\hbbeta^{\rm R}_n)^{(t)}=(\hbbeta_n)^{(t-1)}-(\bC+\bSigma_{\lambda}(\hbbeta^{t-1}))^{-1}\bR^\top\left[ \bR (\bC+\bSigma_{\lambda}(\hbbeta^{(t-1)}))^{-1}\bR^\top\right]^{-1} \left( \bR(\hbbeta_n)^{(t-1)}-\br \right), 
    \end{eqnarray*}
    where $(\hbbeta_n)^{(t-1)}=\left[\bC_n+\bSigma_\lambda(\hbbeta^{t-1})\right]^{-1}\bX_n^\top\bY$ and $\bSigma_{\lambda}(\hbbeta^{(t-1)}) =\frac{\lambda q}{2} \diag\left( \left| \hat\beta_1^{(t-1)} \right|^{q-2},\dots,\left| \hat\beta_p^{(t-1)} \right|^{q-2}  \right)$ for $t=1,2,\dots$. The algorithm stops when $\left\|(\hbbeta^{\rm R}_n)^{(t)}-(\hbbeta^{\rm R}_n)^{(t-1)}  \right\|<\eta$. During the iteration if $\left| \hat\beta_j^{(t-1)} \right| < \eta$, then we delete $j^{th}$ variable and exclude it from final model to make algorithm stable.
\end{algorithmic}
\end{algorithm}
\begin{remark}
In Algorithm 1, we used the ridge coefficients for Step 2 and $\eta=10^{-7}$.
\end{remark}
\begin{remark}
Algorithm 1 is just for the computation of RBRIDGE estimator. In our numerical studies, we will also consider the estimation of \ref{eq-2.3} following \cite{park2011bridge} for comparison sake. 
\end{remark}

\section{Simulation}
\label{sec:sim}
 We generated response from the following model
\begin{equation*}
y_i = \bx^{\top}_i\bbeta+\sigma\epsilon_i, \quad 1\le i \le n,      
\end{equation*}
where $\bx_i \in \mathbb{R}^p$ are zero mean multivariate normal random vectors with correlation matrix $\bSigma=(\Sigma_{ij})$ with $\Sigma_{ij}=\rho^{|i-j|}$ and $\epsilon_i$ are standard normal. We consider $\sigma=1,3$ and $\rho=0.5,0.9$. There are two examples as follows:

\begin{itemize}[align=left]

\item[Ex 1] Following \cite{fan2001variable,tibshirani1996regression}, we consider the true parameters as $\bbeta = (3, 1.5, 0, 0, 2, 0, 0, 0)^{\top}$ and generated data sets consisting of $n\in\{40,60\}$ observations. For the rest of simulation procedure we consider the following four scenarios about the specification of $\bR$ and $\br$ matrices.  
%\begin{itemize}[align=left]
\subitem Case 1- Let $\bR = \left[ 1,1,0,0,1,0,0,0 \right] $ and $\br = \left[ 6.5 \right] $, that is, $\beta_1+\beta_2+\beta_5 = 6.5$. 
\subitem Case 2- Let  $\bR = \left[ -1,1,0,0,1,0,0,0 \right] $ and $\br = \left[ 0.5 \right] $, that is, $\beta_2+\beta_5 = \beta_1 + 0.5$. 
\subitem Case 3- We consider both cases (i) and (ii) simultaneously, that is, \linebreak $\bR = \left[ \begin{matrix}  1&1&0&0&1&0&0&0 \\  -1&1&0&0&1&0&0&0  \end{matrix} \right] $ and 
$  \br = \begin{bmatrix}
  6.5 \\
  0.5 \\
  \end{bmatrix}$.
\subitem Case 4- Let $\bbeta =  \left( \bbeta_1^\top,\bbeta_2^\top \right)^\top$, where $\bbeta_1$ presents the vector of non-zero variables of $\bbeta$ while $\bbeta_2$ presents zeros. As a special case of general form of null hypothesis $\bR\bbeta = \br$, we consider $\bR = \left[\begin{matrix} 0 & 0 & 1 & 0 & 0 & 0 & 0 & 0 \\ 0 & 0 & 0 & 1 & 0 & 0 & 0 & 0 \\ 0 & 0 & 0 & 0 & 0 & 1 & 0 & 0 \\ 0 & 0 & 0 & 0 & 0 & 0 & 1 & 0 \\ 0 & 0 & 0 & 0 & 0 & 0 & 0 & 1 \end{matrix}\right]$ and $\br = \begin{bmatrix}
  0 \\
  0 \\
  0 \\
  0 \\
  0 \\
  \end{bmatrix}$, that is, $\bbeta_2=\begin{bmatrix}
  0 \\
  0 \\
  0 \\
  0 \\
  0 \\
  \end{bmatrix}$.
\item[Ex 2] This example is devoted to the situations where the number of co-variates is larger than the number of observations, which We only consider $n=50$ and $p=(100,200)$. The vector of true parameters is then taken as follows:
\begin{equation*}
\bbeta =\left( \underbrace { 0,\dots,0}_{10},\underbrace {2,\dots,2}_{10}, \underbrace { 0,\dots,0}_{10},
\underbrace { -2,\dots,-2}_{10},\underbrace {0,\dots,0}_{p-40}\right)^\top. 
\end{equation*}
Let again $\bbeta =  \left( \bbeta_1^\top,\bbeta_2^\top \right)^\top$, where $\bbeta_1$ presents the vector of non-zero variables of $\bbeta$ while $\bbeta_2$ presents zeros. Hence, we consider the following four cases.

\subitem{Case 1- } We consider an $\bR=[\bzero,\bI]$ matrix where $\bzero$ and $\bI$ are suitable sizes zero and identity matrices, respectively, such that $\bbeta_2 = \br = {\bf 1}_{p-20}^{\top}\cdot 0$.
\subitem{Case 2- } Similar to the Case 1, except $\br = {\bf 1}_{p-20}^{\top}\cdot 0.1$. In this case, we investigate violations of sub-model in Case 1.

\subitem{Case 3- } Let $\bR=[\bI,\bzero]$, where $\bzero$ and $\bI$ are suitable sizes zero and identity matrices, respectively, such that $\bbeta_1 = \br = ({\bf 1}_{10}\cdot 2, {\bf 1}_{10}\cdot -2)^{\top}$.
\subitem{Case 4- } This is similar to Case 3, except that $\br = ({\bf 1}_{10}\cdot 2.1, {\bf 1}_{10}\cdot -2.1)^{\top}$. In this case, we investigate violations of sub-model in Case 3.
\end{itemize}

We also use the following criteria to asses the numerical performance:
\begin{itemize}[align=left]
\item[\bf MME] presents the median of the model error (ME) measure of an estimator, where ${\rm ME}=(\hbbeta-\bbeta)^{\top}\Sigma(\hbbeta-\bbeta)$.
\item[\bf C] shows the average number of zero coefficients correctly estimated to be zero.
\item[\bf IC] shows the average number of nonzero coefficients incorrectly estimated to be zero.
\item[\bf U-fit] (Under fit) shows the proportion of excluding any significant variables.
\item[\bf C-fit] (Correct fit) presents the probability of selecting the exact subset
model.
\item[\bf O-fit] (Over fit) shows the probability of including all three significant variables and some noise variables.
\end{itemize}

In our simulation study, we compare the performance of the RBRIDGE estimators with LASSO, RIDGE, Elastic Net (E-NET), SCAD, and ORACLE, which the latter is the ordinary least squares estimator of the true model, i.e., $y=\beta_1x_1+\beta_2x_2+\beta_4x_4$ for Example 1. For better specification the E-NET penalty term has form
$\alpha \beta^2 + (1-\alpha)|\beta|$.
So, in our numerical analysis we consider LASSO($\alpha=1$), RIDGE($\alpha=0$) and E-NET($\alpha=0.2,0.4,0.6,0.8$ which gives minimum prediction error) using the {\it glmnet} package in R. For the SCAD we use the {\it ncvreg} package in R. Note that both the BRIDGE and RBRIDGE estimators are calculated using the {\it rbridge} package which will be appeared online soon. Also, a grid of values for $q$, from $0.25$ to $2$ with $0.25$ increment is taken, in addition to the one which gives the minimum measurement error.

%%%%%%%%%%%%%%%%%%%%%%%%%%%%%%%%%%%%%%%%%%%%%%%%%%%%%%%%%%%%
%%%%% Example 1 %%%%%%%%%%%%%%%%%%%%%%%%%%%%%%%%%%%%%%%%%%%%%%%%%
%%%%%%%%%%%%%%%%%%%%%%%%%%%%%%%%%%%%%%%%%%%%%%%%%%%%%%%%%%%%
\begin{table}[H]
\caption{Simulation Results for Example 1
\label{tab:sim:ex1}}
\centering
\begin{adjustbox}{width=.95\textwidth}
\begin{tabular}{lcccccccccccc}
\toprule
\multicolumn{1}{c}{ } & \multicolumn{6}{c}{$n=40$, $\sigma=1$, $\rho=0.5$} & \multicolumn{6}{c}{$n=40$, $\sigma=1$, $\rho=0.9$} \\
\cmidrule(l{3pt}r{3pt}){2-7} \cmidrule(l{3pt}r{3pt}){8-13}
  & MME & C & IC & U-fit & C-fit & O-fit & MME & C & IC & U-fit & C-fit & O-fit\\
\midrule
LASSO & 0.154 & 2.768 & 0.000 & 0.000 & 0.118 & 0.882 & 0.461 & 2.888 & 0.000 & 0.000 & 0.102 & 0.898\\
RIDGE & 0.296 & 0.000 & 0.000 & 0.000 & 0.000 & 1.000 & 1.175 & 0.000 & 0.000 & 0.000 & 0.000 & 1.000\\
E-NET & 0.174 & 2.104 & 0.000 & 0.000 & 0.068 & 0.932 & 0.567 & 2.298 & 0.000 & 0.000 & 0.046 & 0.954\\
SCAD & 0.082 & 4.444 & 0.000 & 0.000 & 0.756 & 0.244 & 0.554 & 3.970 & 0.172 & 0.170 & 0.420 & 0.410\\
ORACLE & 0.064 & 5.000 & 0.000 & 0.000 & 1.000 & 0.000 & 0.156 & 5.000 & 0.000 & 0.000 & 1.000 & 0.000\\
BRIDGE & 0.097 & 3.230 & 0.000 & 0.000 & 0.000 & 1.000 & 0.437 & 3.118 & 0.016 & 0.016 & 0.152 & 0.832\\
RBRIDGE$^1$ & 0.160 & 2.358 & 0.000 & 0.000 & 0.268 & 0.732 & 0.551 & 1.706 & 0.010 & 0.010 & 0.156 & 0.834\\
RBRIDGE$^2$ & 0.103 & 3.154 & 0.000 & 0.000 & 0.000 & 1.000 & 0.172 & 3.220 & 0.000 & 0.000 & 0.022 & 0.978\\
RBRIDGE$^3$ & 0.166 & 1.576 & 0.000 & 0.000 & 0.126 & 0.874 & 0.298 & 1.290 & 0.000 & 0.000 & 0.056 & 0.944\\
RBRIDGE$^4$ & 0.063 & 5.000 & 0.000 & 0.000 & 1.000 & 0.000 & 0.170 & 5.000 & 0.000 & 0.000 & 1.000 & 0.000\\
\addlinespace
\multicolumn{1}{c}{ } & \multicolumn{6}{c}{$n=40$, $\sigma=3$, $\rho=0.5$} & \multicolumn{6}{c}{$n=40$, $\sigma=3$, $\rho=0.9$} \\
\cmidrule(l{3pt}r{3pt}){2-7} \cmidrule(l{3pt}r{3pt}){8-13}
  & MME & C & IC & U-fit & C-fit & O-fit & MME & C & IC & U-fit & C-fit & O-fit\\
\midrule
LASSO & 1.374 & 2.796 & 0.034 & 0.034 & 0.122 & 0.844 & 3.615 & 2.950 & 0.416 & 0.386 & 0.082 & 0.532\\
RIDGE & 1.783 & 0.000 & 0.000 & 0.000 & 0.000 & 1.000 & 3.251 & 0.000 & 0.000 & 0.000 & 0.000 & 1.000\\
E-NET & 1.455 & 2.186 & 0.018 & 0.018 & 0.068 & 0.914 & 3.461 & 2.354 & 0.258 & 0.246 & 0.040 & 0.714\\
SCAD & 1.770 & 3.352 & 0.172 & 0.166 & 0.166 & 0.668 & 7.120 & 3.622 & 1.110 & 0.804 & 0.020 & 0.176\\
ORACLE & 0.579 & 5.000 & 0.000 & 0.000 & 1.000 & 0.000 & 1.407 & 5.000 & 0.000 & 0.000 & 1.000 & 0.000\\
BRIDGE & 1.697 & 2.988 & 0.086 & 0.086 & 0.222 & 0.692 & 3.926 & 2.198 & 0.462 & 0.384 & 0.076 & 0.540\\
RBRIDGE$^1$ & 0.682 & 1.162 & 0.022 & 0.022 & 0.116 & 0.862 & 1.297 & 1.500 & 0.182 & 0.180 & 0.076 & 0.744\\
RBRIDGE$^2$ & 0.820 & 3.272 & 0.010 & 0.010 & 0.184 & 0.806 & 4.218 & 2.754 & 0.192 & 0.190 & 0.254 & 0.556\\
RBRIDGE$^3$ & 0.339 & 0.990 & 0.000 & 0.000 & 0.080 & 0.920 & 0.592 & 1.486 & 0.054 & 0.054 & 0.130 & 0.816\\
RBRIDGE$^4$ & 0.563 & 5.000 & 0.006 & 0.006 & 0.994 & 0.000 & 1.512 & 5.000 & 0.048 & 0.048 & 0.952 & 0.000\\
\addlinespace
\multicolumn{1}{c}{ } & \multicolumn{6}{c}{$n=60$, $\sigma=1$, $\rho=0.5$} & \multicolumn{6}{c}{$n=60$, $\sigma=1$, $\rho=0.9$} \\
\cmidrule(l{3pt}r{3pt}){2-7} \cmidrule(l{3pt}r{3pt}){8-13}
  & MME & C & IC & U-fit & C-fit & O-fit & MME & C & IC & U-fit & C-fit & O-fit\\
\midrule
LASSO & 0.094 & 2.696 & 0.000 & 0.000 & 0.102 & 0.898 & 0.267 & 2.916 & 0.000 & 0.000 & 0.086 & 0.914\\
RIDGE & 0.211 & 0.000 & 0.000 & 0.000 & 0.000 & 1.000 & 1.046 & 0.000 & 0.000 & 0.000 & 0.000 & 1.000\\
E-NET & 0.104 & 2.146 & 0.000 & 0.000 & 0.056 & 0.944 & 0.328 & 2.432 & 0.000 & 0.000 & 0.050 & 0.950\\
SCAD & 0.050 & 4.480 & 0.000 & 0.000 & 0.770 & 0.230 & 0.178 & 4.270 & 0.050 & 0.050 & 0.634 & 0.316\\
ORACLE & 0.039 & 5.000 & 0.000 & 0.000 & 1.000 & 0.000 & 0.098 & 5.000 & 0.000 & 0.000 & 1.000 & 0.000\\
BRIDGE & 0.057 & 3.114 & 0.000 & 0.000 & 0.000 & 1.000 & 0.188 & 3.220 & 0.004 & 0.004 & 0.070 & 0.926\\
RBRIDGE$^1$ & 0.108 & 2.684 & 0.000 & 0.000 & 0.338 & 0.662 & 0.353 & 1.842 & 0.002 & 0.002 & 0.172 & 0.826\\
RBRIDGE$^2$ & 0.070 & 3.008 & 0.000 & 0.000 & 0.000 & 1.000 & 0.092 & 3.190 & 0.000 & 0.000 & 0.014 & 0.986\\
RBRIDGE$^3$ & 0.122 & 1.626 & 0.000 & 0.000 & 0.140 & 0.860 & 0.203 & 1.434 & 0.000 & 0.000 & 0.088 & 0.912\\
RBRIDGE$^4$ & 0.040 & 5.000 & 0.000 & 0.000 & 1.000 & 0.000 & 0.103 & 5.000 & 0.000 & 0.000 & 1.000 & 0.000\\
\addlinespace
\multicolumn{1}{c}{ } & \multicolumn{6}{c}{$n=60$, $\sigma=3$, $\rho=0.5$} & \multicolumn{6}{c}{$n=60$, $\sigma=3$, $\rho=0.9$} \\
\cmidrule(l{3pt}r{3pt}){2-7} \cmidrule(l{3pt}r{3pt}){8-13}
  & MME & C & IC & U-fit & C-fit & O-fit & MME & C & IC & U-fit & C-fit & O-fit\\
\midrule
LASSO & 0.851 & 2.696 & 0.004 & 0.004 & 0.102 & 0.894 & 2.266 & 2.864 & 0.216 & 0.208 & 0.076 & 0.716\\
RIDGE & 1.095 & 0.000 & 0.000 & 0.000 & 0.000 & 1.000 & 2.497 & 0.000 & 0.000 & 0.000 & 0.000 & 1.000\\
E-NET & 0.907 & 2.246 & 0.002 & 0.002 & 0.058 & 0.940 & 2.356 & 2.346 & 0.142 & 0.136 & 0.032 & 0.832\\
SCAD & 0.918 & 3.478 & 0.068 & 0.068 & 0.246 & 0.686 & 3.822 & 3.676 & 0.838 & 0.686 & 0.062 & 0.252\\
ORACLE & 0.355 & 5.000 & 0.000 & 0.000 & 1.000 & 0.000 & 0.883 & 5.000 & 0.000 & 0.000 & 1.000 & 0.000\\
BRIDGE & 0.844 & 3.110 & 0.030 & 0.030 & 0.146 & 0.824 & 2.744 & 2.240 & 0.304 & 0.272 & 0.132 & 0.596\\
RBRIDGE$^1$ & 0.562 & 1.358 & 0.004 & 0.004 & 0.124 & 0.872 & 0.901 & 1.254 & 0.124 & 0.124 & 0.076 & 0.800\\
RBRIDGE$^2$ & 0.466 & 3.182 & 0.000 & 0.000 & 0.056 & 0.944 & 2.495 & 2.860 & 0.098 & 0.098 & 0.244 & 0.658\\
RBRIDGE$^3$ & 0.265 & 1.004 & 0.000 & 0.000 & 0.056 & 0.944 & 0.380 & 1.366 & 0.020 & 0.020 & 0.106 & 0.874\\
RBRIDGE$^4$ & 0.351 & 5.000 & 0.000 & 0.000 & 1.000 & 0.000 & 0.965 & 5.000 & 0.010 & 0.010 & 0.990 & 0.000\\
\bottomrule
\end{tabular}
\end{adjustbox}
\begin{flushleft}                           
\footnotesize{$^1$the restrictions in case 1, $^2$the restrictions in case 2,
$^3$the restrictions in case 3, $^4$the restrictions in case 4.}
\end{flushleft}
\end{table}

In Table~\ref{tab:sim:ex1}, we report the the result obtained from the simulation study of Example 1 and briefly summarize as follows: if the sample size is small, the noise level is low and the the effect of multicollinearity is moderate, the SCAD performs well compared to the LASSO, RIDGE and E-NET and it reduces both ME and model complexity (MC). If the noise level is high with same the $n$ and $\rho$, the LASSO becomes better which is consistent with the results of \cite{fan2001variable}. Apart from their results, we notice that SCAD loses its efficiency in terms of both the ME and MC as $\rho$ gets large. RIDGE may only reduce ME and does not reduce MC since it does not shrink coefficients to zero. The E-NET has advantage compared to the RIDGE since it shrinks coefficients to zero even if it has larger ME compared to the RIDGE. With this analysis in hand, we now concentrate on performance evaluation of the proposed $L_q$-type estimation. We first note that the performance of the RBRIDGE$^4$ is mostly close to the ORACLE which means it reduces both the ME and MC. When the noise level is low, the RBRIDGE$^2$, which is the Case 2 of restrictions, has better performance compared to Cases 1 and 3. On the other hand, as the noise level gets larger, the RBRIDGE$^3$ that is estimated based on the combination of both the restrictions in Cases 1 \& 2 reduces the ME more than that of the RBRIDGE$^1$ and RBRIDGE$^2$. Also, the MMEs of the proposed estimators become generally smaller as $n$ increases.

\begin{figure}[H]
\centering
\includegraphics[width=16cm,height=16cm]{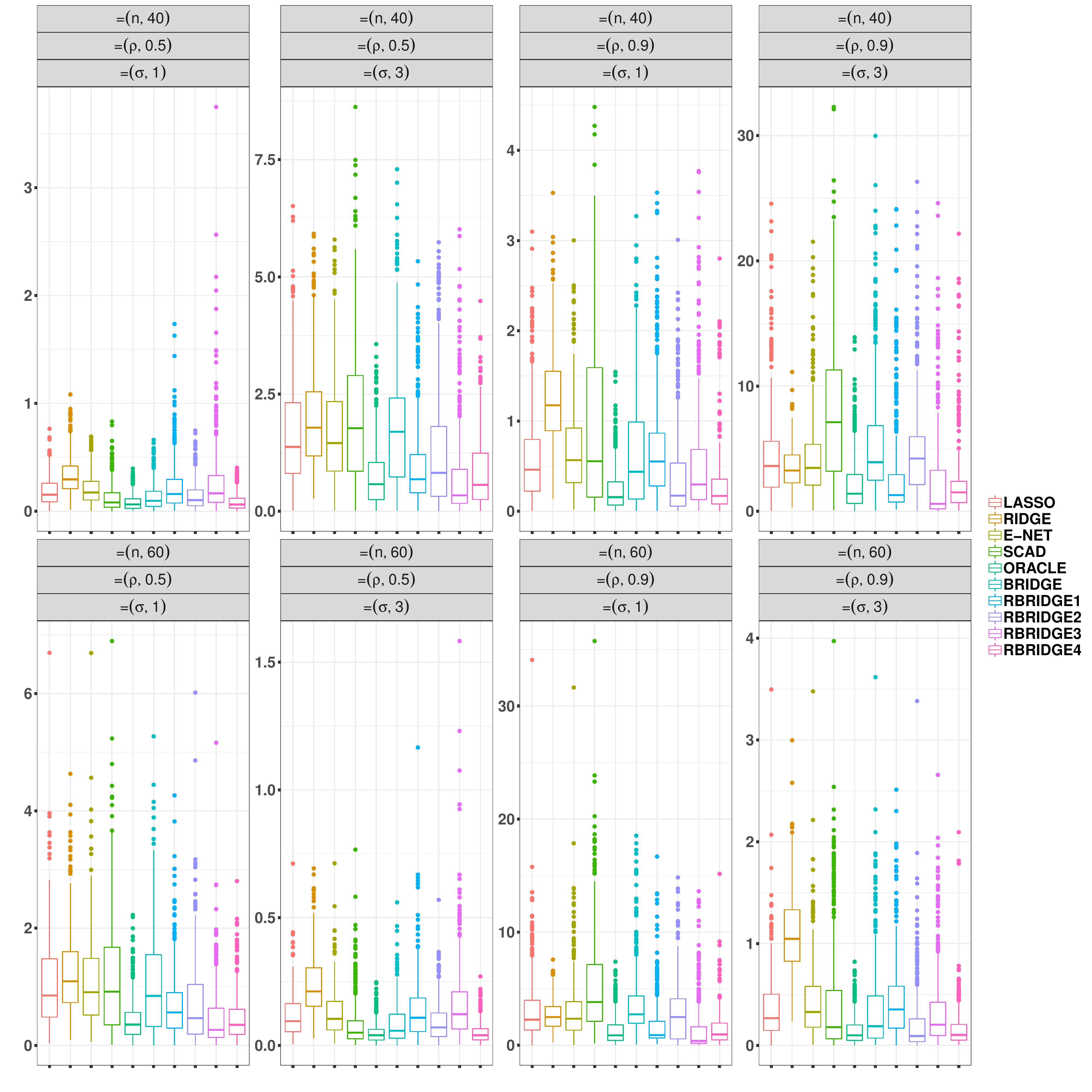}
\caption{Boxplots of the MEs for Example 1
 \label{Fig:boxplot1:ex1}}
\end{figure}

Among all, the RBRIDGE$^4$ is preferred with respect to the others measures, C, IC, U-fit, C-fit, and O-fit.

%%%%%%%%%%%%%%%%%%%%%%%%%%%%%%%%%%%%%%%%%%%%%%%%%%%%%%%%%%%%
%%%%% Example 2 %%%%%%%%%%%%%%%%%%%%%%%%%%%%%%%%%%%%%%%%%%%%%%%%%
%%%%%%%%%%%%%%%%%%%%%%%%%%%%%%%%%%%%%%%%%%%%%%%%%%%%%%%%%%%%
\begin{table}[H]
\caption{Simulation Results for Example 2
\label{tab:sim:ex2}}
\centering
\begin{adjustbox}{width=.95\textwidth}
\begin{tabular}{lrrrrrrrrrrrr}
\toprule
\multicolumn{1}{c}{ } & \multicolumn{6}{c}{$p=100$, $\sigma=1$, $\rho=0.5$} & \multicolumn{6}{c}{$p=100$, $\sigma=1$, $\rho=0.9$} \\
\cmidrule(l{3pt}r{3pt}){2-7} \cmidrule(l{3pt}r{3pt}){8-13}
  & MME & C & IC & U-fit & C-fit & O-fit & MME & C & IC & U-fit & C-fit & O-fit\\
\midrule
LASSO & 3.091 & 63.74 & 0.08 & 0.06 & 0.0 & 0.94 & 1.520 & 74.46 & 0.12 & 0.12 & 0.00 & 0.88\\
RIDGE & 114.298 & 0.00 & 0.00 & 0.00 & 0.0 & 1.00 & 127.445 & 0.00 & 0.00 & 0.00 & 0.00 & 1.00\\
E-NET & 3.537 & 60.42 & 0.04 & 0.04 & 0.0 & 0.96 & 0.994 & 72.66 & 0.00 & 0.00 & 0.00 & 1.00\\
SCAD & 61.240 & 73.34 & 8.66 & 1.00 & 0.0 & 0.00 & 24.683 & 77.20 & 14.00 & 1.00 & 0.00 & 0.00\\
ORACLE & 0.527 & 80.00 & 0.00 & 0.00 & 1.0 & 0.00 & 0.663 & 80.00 & 0.00 & 0.00 & 1.00 & 0.00\\
BRIDGE & 1.106 & 70.70 & 0.08 & 0.04 & 0.1 & 0.86 & 1.554 & 33.56 & 0.42 & 0.18 & 0.00 & 0.82\\
RBRIDGE$^1$ & 0.497 & 80.00 & 0.00 & 0.00 & 1.0 & 0.00 & 0.317 & 80.00 & 0.00 & 0.00 & 1.00 & 0.00\\
RBRIDGE$^2$ & 4.110 & 0.00 & 0.00 & 0.00 & 0.0 & 1.00 & 11.480 & 0.00 & 0.40 & 0.06 & 0.00 & 0.94\\
RBRIDGE$^3$ & 0.094 & 47.76 & 0.00 & 0.00 & 0.0 & 1.00 & 0.107 & 43.12 & 0.00 & 0.00 & 0.00 & 1.00\\
RBRIDGE$^4$ & 0.840 & 46.94 & 0.00 & 0.00 & 0.0 & 1.00 & 1.221 & 46.08 & 0.00 & 0.00 & 0.00 & 1.00\\
\addlinespace
\multicolumn{1}{c}{ } & \multicolumn{6}{c}{$p=100$, $\sigma=3$, $\rho=0.5$} & \multicolumn{6}{c}{$p=100$, $\sigma=3$, $\rho=0.9$} \\
\cmidrule(l{3pt}r{3pt}){2-7} \cmidrule(l{3pt}r{3pt}){8-13}
  & MME & C & IC & U-fit & C-fit & O-fit & MME & C & IC & U-fit & C-fit & O-fit\\
\midrule
LASSO & 21.047 & 63.16 & 1.44 & 0.66 & 0.0 & 0.34 & 7.473 & 71.22 & 3.18 & 0.98 & 0.00 & 0.02\\
RIDGE & 113.484 & 0.00 & 0.00 & 0.00 & 0.0 & 1.00 & 127.367 & 0.00 & 0.00 & 0.00 & 0.00 & 1.00\\
E-NET & 19.890 & 57.24 & 0.76 & 0.44 & 0.0 & 0.56 & 5.884 & 69.00 & 1.10 & 0.56 & 0.00 & 0.44\\
SCAD & 74.389 & 74.30 & 9.70 & 1.00 & 0.0 & 0.00 & 29.122 & 77.12 & 14.20 & 1.00 & 0.00 & 0.00\\
ORACLE & 4.744 & 80.00 & 0.00 & 0.00 & 1.0 & 0.00 & 5.971 & 80.00 & 0.00 & 0.00 & 1.00 & 0.00\\
BRIDGE & 24.677 & 28.18 & 1.52 & 0.40 & 0.0 & 0.60 & 7.863 & 23.84 & 2.00 & 0.32 & 0.00 & 0.68\\
RBRIDGE$^1$ & 3.821 & 80.00 & 0.00 & 0.00 & 1.0 & 0.00 & 1.694 & 80.00 & 0.26 & 0.04 & 0.96 & 0.00\\
RBRIDGE$^2$ & 6.817 & 0.00 & 0.00 & 0.00 & 0.0 & 1.00 & 12.274 & 0.00 & 0.74 & 0.08 & 0.00 & 0.92\\
RBRIDGE$^3$ & 0.506 & 35.00 & 0.00 & 0.00 & 0.0 & 1.00 & 0.531 & 46.00 & 0.00 & 0.00 & 0.00 & 1.00\\
RBRIDGE$^4$ & 1.286 & 46.56 & 0.00 & 0.00 & 0.0 & 1.00 & 1.969 & 38.24 & 0.00 & 0.00 & 0.00 & 1.00\\
\addlinespace
\multicolumn{1}{c}{ } & \multicolumn{6}{c}{$p=200$, $\sigma=1$, $\rho=0.5$} & \multicolumn{6}{c}{$p=200$, $\sigma=1$, $\rho=0.9$} \\
\cmidrule(l{3pt}r{3pt}){2-7} \cmidrule(l{3pt}r{3pt}){8-13}
  & MME & C & IC & U-fit & C-fit & O-fit & MME & C & IC & U-fit & C-fit & O-fit\\
\midrule
LASSO & 43.233 & 162.98 & 4.00 & 0.86 & 0.0 & 0.14 & 2.885 & 171.66 & 0.64 & 0.42 & 0.00 & 0.58\\
RIDGE & 139.543 & 0.00 & 0.00 & 0.00 & 0.0 & 1.00 & 167.377 & 0.00 & 0.00 & 0.00 & 0.00 & 1.00\\
E-NET & 40.517 & 155.14 & 2.46 & 0.76 & 0.0 & 0.24 & 2.061 & 168.50 & 0.08 & 0.08 & 0.00 & 0.92\\
SCAD & 117.997 & 175.56 & 12.10 & 1.00 & 0.0 & 0.00 & 40.314 & 175.20 & 14.98 & 1.00 & 0.00 & 0.00\\
ORACLE & 0.959 & 180.00 & 0.00 & 0.00 & 1.0 & 0.00 & 0.974 & 180.00 & 0.00 & 0.00 & 1.00 & 0.00\\
BRIDGE & 46.456 & 121.66 & 3.34 & 0.58 & 0.0 & 0.42 & 3.577 & 157.72 & 1.40 & 0.58 & 0.00 & 0.42\\
RBRIDGE$^1$ & 0.936 & 180.00 & 0.00 & 0.00 & 1.0 & 0.00 & 0.398 & 180.00 & 0.00 & 0.00 & 1.00 & 0.00\\
RBRIDGE$^2$ & 9.470 & 0.00 & 0.00 & 0.00 & 0.0 & 1.00 & 33.882 & 0.00 & 0.46 & 0.04 & 0.00 & 0.96\\
RBRIDGE$^3$ & 0.311 & 102.98 & 0.00 & 0.00 & 0.0 & 1.00 & 0.140 & 93.24 & 0.00 & 0.00 & 0.00 & 1.00\\
RBRIDGE$^4$ & 1.294 & 66.74 & 0.00 & 0.00 & 0.0 & 1.00 & 1.770 & 114.22 & 0.00 & 0.00 & 0.00 & 1.00\\
\addlinespace
\multicolumn{1}{c}{ } & \multicolumn{6}{c}{$p=200$, $\sigma=3$, $\rho=0.5$} & \multicolumn{6}{c}{$p=200$, $\sigma=3$, $\rho=0.9$} \\
\cmidrule(l{3pt}r{3pt}){2-7} \cmidrule(l{3pt}r{3pt}){8-13}
  & MME & C & IC & U-fit & C-fit & O-fit & MME & C & IC & U-fit & C-fit & O-fit\\
\midrule
LASSO & 56.255 & 163.62 & 5.76 & 0.96 & 0.0 & 0.04 & 13.064 & 167.04 & 4.74 & 1.00 & 0.00 & 0.00\\
RIDGE & 139.681 & 0.00 & 0.00 & 0.00 & 0.0 & 1.00 & 167.472 & 0.00 & 0.00 & 0.00 & 0.00 & 1.00\\
E-NET & 53.906 & 152.80 & 3.62 & 0.90 & 0.0 & 0.10 & 11.059 & 163.66 & 1.84 & 0.70 & 0.00 & 0.30\\
SCAD & 118.643 & 176.00 & 12.78 & 1.00 & 0.0 & 0.00 & 43.988 & 174.40 & 15.36 & 1.00 & 0.00 & 0.00\\
ORACLE & 8.633 & 180.00 & 0.00 & 0.00 & 1.0 & 0.00 & 8.769 & 180.00 & 0.00 & 0.00 & 1.00 & 0.00\\
BRIDGE & 62.173 & 89.22 & 3.14 & 0.52 & 0.0 & 0.48 & 15.952 & 117.10 & 5.38 & 0.66 & 0.00 & 0.34\\
RBRIDGE$^1$ & 5.592 & 180.00 & 0.00 & 0.00 & 1.0 & 0.00 & 1.987 & 180.00 & 0.28 & 0.04 & 0.96 & 0.00\\
RBRIDGE$^2$ & 13.364 & 0.00 & 0.12 & 0.02 & 0.0 & 0.98 & 35.756 & 0.00 & 1.46 & 0.14 & 0.00 & 0.86\\
RBRIDGE$^3$ & 2.503 & 49.80 & 0.00 & 0.00 & 0.0 & 1.00 & 0.433 & 73.88 & 0.00 & 0.00 & 0.00 & 1.00\\
RBRIDGE$^4$ & 3.202 & 50.12 & 0.00 & 0.00 & 0.0 & 1.00 & 2.674 & 66.62 & 0.00 & 0.00 & 0.00 & 1.00\\
\bottomrule
\end{tabular}
\end{adjustbox}
\begin{flushleft}                           
\footnotesize{$^1$the restrictions in case 1, $^2$the restrictions in case 2,
$^3$the restrictions in case 3, $^4$the restrictions in case 4.}
\end{flushleft}
\end{table}

As the noise level increases its performance violates, however, it is still the best amongst. On the other hand, the O-fit of other RBRIDGE estimators are higher than the rest. However, the O-fit decreases as noise level increases. Eventually, as the level of multicollinearity increases the C-fit measure of the proposed estimators decreases, which is compatible with other penalty estimators. In the Figure~\ref{Fig:boxplot1:ex1}, we plot MEs, which the figures outputs agree with the reported analyses. 

\begin{figure}[H]
\centering
\includegraphics[width=16cm,height=16cm]{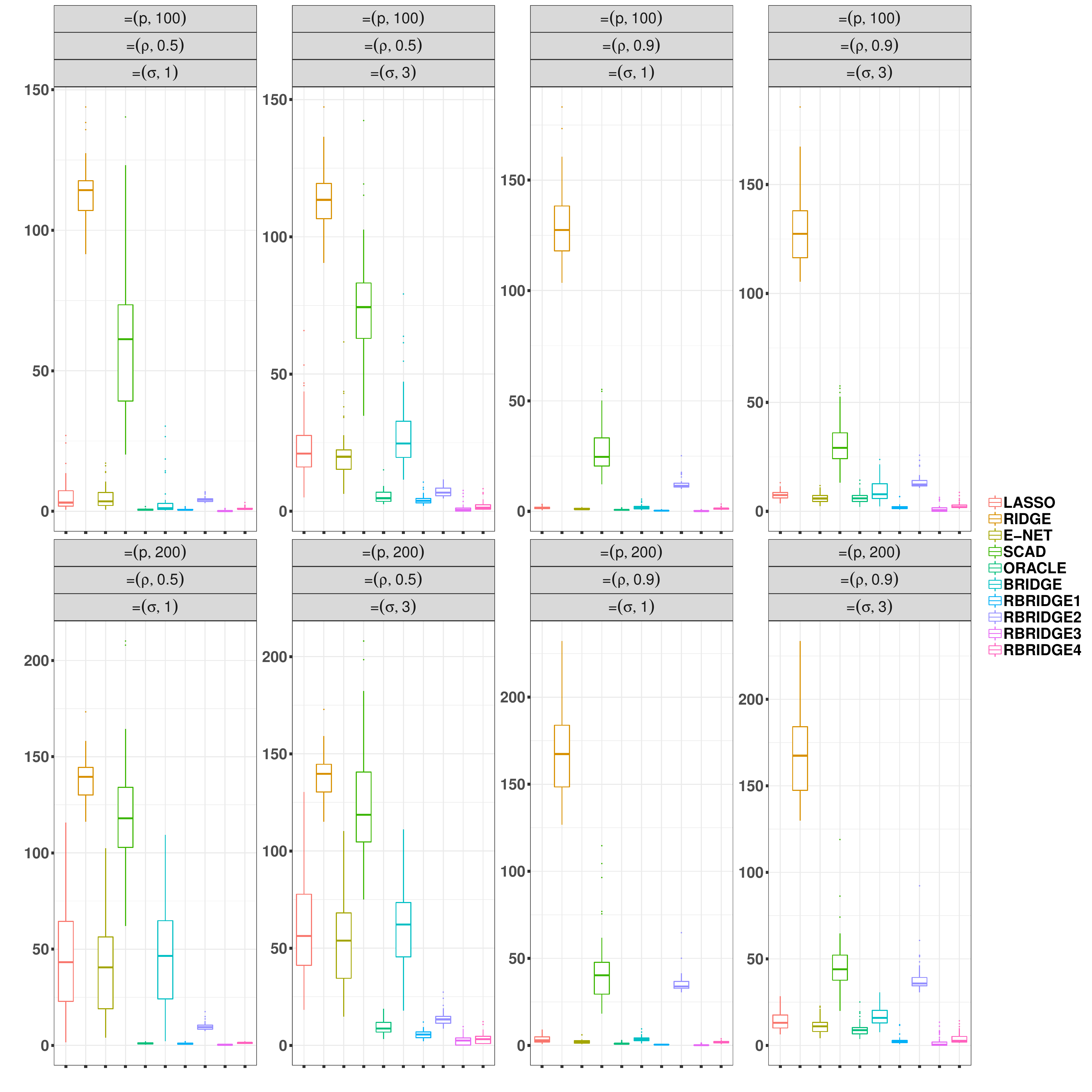}
\caption{Boxplots of the MEs for Example 2.
 \label{Fig:boxplot1:ex2}}
\end{figure}

Similar to Table~\ref{tab:sim:ex1}, Table~\ref{tab:sim:ex2} provides the summary for the high dimensional case in which the number of co-variates exceeds the number of samples. Obviously both ME and MC measures increase in this case, however, as the sample size increases the values decrease that show consistency in estimation. As the noise level increases the performance of the proposed RBRIDGE estimators is as good as the oracle, however, the other penalty estimators perform poorly. To be more specific, it is clearly seen that both the RIDGE and SCAD perform very poor in terms of the ME and MC in most cases. 
This gets much worse if $p$ increases. In terms of estimation 
accuracy the RBRIDGE estimators perform closely to the oracle estimator, whereas our proposed approach gives the smallest MME, and consistently outperforms other penalty estimators.
In terms of variable selection we observe that the well-known penalty estimators does not lead to a sparse model with restriction, whereas, the RBRIDGE estimators successfully select all covariates with nonzero coefficients, but it is obvious that the proposed RBRIDGE$^2$ has slightly strange sparsity rate (i.e., zero Cs) than the other RBRIDGE estimators.  Figure~\ref{Fig:boxplot1:ex2}, we plot MEs; again confirm the performance analysis outlined in above.

%%%%%%%%%%%%%%%%%%%%%%%%%%%%%%%%%%%%%%%%%%%%%%%%%%%%%%%%%%%%
%%%%%%%%%%%%%%%%%%Real Data  %%%%%%%%%%%%%%%%%%%%%%%%%%%%%%%
%%%%%%%%%%%%%%%%%%%%%%%%%%%%%%%%%%%%%%%%%%%%%%%%%%%%%%%%%%%%
\section{Real Data Analyses}
\label{sec:RDA}
We exhaustively consider four real data examples which cover both low dimensional and high dimensional cases. Each data set is freely available online and can be downloaded from the given references. In all applications, we analyze the performance of the respective approaches in the simulation. Before we start analyzing, the response variable is centered and the predictors are standardized. Hence, a constant term is not counted as a parameter. To asses the prediction accuracy of the listed estimators, we randomly split the data into two equal part of observations, the first part is the training set and the other part is for test set. Models are fitted on the training set only. We consider the following measure to asses the performance of the estimators. 
\begin{equation}
\label{mse_y}
\rm MSE_y =\left\| \by_{test}-\bX_{test}\hbbeta \right\|^2    
\end{equation}
We also consider the following measure in the case where a prior information $\bbeta_{0i}$, $i=1,2$, is available. 
 \begin{equation}
\label{mse_beta}
\rm MSE_{\bbeta_{0i}} =\left\| \bbeta_{0i}-\hbbeta \right\|^2, i=1,2.
\end{equation}
For clarity, here $i$ takes two values since we consider two sets of prior information (restrictions) in our numerical analyses. 

This process is repeated $500$ times if $n>p$, otherwise $100$, and the median values are reported. For the ease of comparison, we calculate the $\rm RMSE = \frac{\rm MSE}{\min \rm MSE}$; values larger than $1$ show the estimator performs worse compared to the one that has the minimum MSE.

\subsection{Air Pollution and Mortality Data}

Air pollution has an impact on health and leads to disease or hospitalization if someone has been exposed to excessive air pollution for a long time. Everyone may be exposed to air pollution.

\begin{table}[!htbp]
%\begin{adjustbox}{width=1\textwidth}
\centering
\begin{tabular}{ll}
\toprule
\textbf{Variables} & \textbf{Descriptions} \\
\midrule
Mortality(Response)    & Total age-adjusted mortality from \cr
&all causes (Annual deaths per 100,000 people) \cr
Precip      & Mean annual precipitation (inches) \cr
Humidity   &Percent relative humidity (annual average at 1:00pm)\cr
JanTemp    &Mean January temperature (degrees F)\cr
JulyTemp   &Mean July temperature (degrees F)\cr
Over65     &Percentage of the population aged 65 years or over\cr
House      &Population per household\cr
Educ       &Median number of school years completed for persons 25 years or older\cr
Sound      &Percentage of the housing that is sound with all facilities\cr
Density    &Population density (in persons per square mile of urbanized area)\cr
NonWhite   &Percentage of population that is nonwhite\cr
WhiteCol   &Percentage of employment in white collar occupations\cr
Poor       &Percentage of households with annual income under \$3,000 in 1960\cr
HC         &Pollution potential of hydrocarbons\cr
NOX        &Pollution potential of oxides of nitrogen\cr
SO2        &Pollution potential of sulfur dioxide\cr
\bottomrule
\end{tabular}
%\end{adjustbox}
\caption{Description of the Air Pollution and Mortality Data variables
\label{Tab:variables}}
\end{table}

Hence, it is worth to investigate what kinds of variable effects mortality. In this example, we use the well-known Air Pollution and Mortality Data which is originally analyzed by \cite{mcdonald1973instabilities} and was also analyzed by \cite{soofi1990effects, smucler2017robust, yuzbacsi2019shrinkage}. The data are freely available at \href{http://lib.stat.cmu.edu/datasets/pollution}{http://lib.stat.cmu.edu/datasets/pollution}. This data consists of 60 sets observations of metropolitan statistical areas in the United States in 1960 on 16 variables, where the description of the variables is provided in Table~\ref{Tab:variables}.

To estimate the RBRIDGE, we need to know a piece of prior information about the real data set. This can be adopted from the previous results, e.g., \cite{mcdonald1973instabilities,yuzbacsi2019shrinkage} or an expert's opinion can be taken. However, we used the stepwise regressions here, making use of the functions {\it ols\_step\_forward\_p} and {\it ols\_step\_backward\_p} in {\it olsrr} package which is recently realized in R. Hence, we have two prior information that are reported in Table~\ref{air_d:beta_t}. According to this table, the ``Stepwise Forward Method" finds that seven variables are significantly important while the other method finds that nine variables are significantly important.
\begin{table}[H]
\caption{True parameters based on the Stepwise regression
\label{air_d:beta_t}
}
\centering
\begin{tabular}{llcc}
  \toprule
  && Stepwise Forward Method & Stepwise Backward Method \\ 
 &Variables& $\bbeta_{01}$ & $\bbeta_{02}$ \\ 
  \midrule
$\beta_1$&  Precip & 16.460 & 18.536 \\ 
$\beta_2$&  Humidity & -- & -- \\ 
$\beta_3$&  JanTemp & -19.256 & -23.005 \\ 
$\beta_4$&  JulyTemp & -10.956 & -15.812 \\ 
$\beta_5$&  Over65 & -- & -15.994 \\ 
$\beta_6$&  House & -8.387 & -18.582 \\ 
$\beta_7$&  Educ & -14.338 & -19.796 \\ 
$\beta_8$&  Sound & -- & -- \\ 
$\beta_9$&  Density & -- & -- \\ 
$\beta_{10}$&  NonWhite & 46.530 & 41.590 \\ 
$\beta_{11}$&  WhiteCol & -- & -- \\ 
$\beta_{12}$&  Poor & -- & -- \\ 
$\beta_{13}$&  HC & -- & -84.819 \\ 
$\beta_{14}$&  NOX & -- & 86.698 \\ 
$\beta_{15}$&  SO2 & 14.284 & -- \\  
   \bottomrule
\end{tabular}
\end{table}
In the light of the prior information in Table~\ref{air_d:beta_t}, the restrictions based on $\bbeta_{01}$ and $\bbeta_{02}$ can be respectively expressed as follows:
\begin{eqnarray}
    \bR_1 = 
\begin{bmatrix}
\begin{tabular}{rrrrrrrrrrrrrrrr}
0 & 1 & 0 & 0 & 0 & 0 & 0 & 0 & 0 & 0 & 0 & 0 & 0 & 0 & 0 \\ 
0 & 0 & 0 & 0 & 1 & 0 & 0 & 0 & 0 & 0 & 0 & 0 & 0 & 0 & 0 \\ 
0 & 0 & 0 & 0 & 0 & 0 & 0 & 1 & 0 & 0 & 0 & 0 & 0 & 0 & 0 \\ 
0 & 0 & 0 & 0 & 0 & 0 & 0 & 0 & 1 & 0 & 0 & 0 & 0 & 0 & 0 \\ 
0 & 0 & 0 & 0 & 0 & 0 & 0 & 0 & 0 & 0 & 1 & 0 & 0 & 0 & 0 \\ 
0 & 0 & 0 & 0 & 0 & 0 & 0 & 0 & 0 & 0 & 0 & 1 & 0 & 0 & 0 \\ 
0 & 0 & 0 & 0 & 0 & 0 & 0 & 0 & 0 & 0 & 0 & 0 & 1 & 0 & 0 \\ 
0 & 0 & 0 & 0 & 0 & 0 & 0 & 0 & 0 & 0 & 0 & 0 & 0 & 1 & 0 
\end{tabular}
  \end{bmatrix},
  & 
 \br_1 = \begin{bmatrix}
  0 \\
  0 \\
  0 \\
  0 \\
  0 \\
  0 \\
  0 \\
  0 \\
  \end{bmatrix},
\end{eqnarray}
that is, $\beta_i=0$, $i=2,5,8,9,11,12,13,14$ and
\begin{eqnarray}
    \bR_2 = 
\begin{bmatrix}
\begin{tabular}{rrrrrrrrrrrrrrrr}
0 & 1 & 0 & 0 & 0 & 0 & 0 & 0 & 0 & 0 & 0 & 0 & 0 & 0 & 0 \\ 
0 & 0 & 0 & 0 & 0 & 0 & 0 & 1 & 0 & 0 & 0 & 0 & 0 & 0 & 0 \\ 
0 & 0 & 0 & 0 & 0 & 0 & 0 & 0 & 1 & 0 & 0 & 0 & 0 & 0 & 0 \\ 
0 & 0 & 0 & 0 & 0 & 0 & 0 & 0 & 0 & 0 & 1 & 0 & 0 & 0 & 0 \\ 
0 & 0 & 0 & 0 & 0 & 0 & 0 & 0 & 0 & 0 & 0 & 1 & 0 & 0 & 0 \\ 
0 & 0 & 0 & 0 & 0 & 0 & 0 & 0 & 0 & 0 & 0 & 0 & 0 & 0 & 1 
\end{tabular}
  \end{bmatrix},
  &  
  \br_2 = \begin{bmatrix}
  0 \\
  0 \\
  0 \\
  0 \\
  0 \\
  0 \\
  \end{bmatrix},
\end{eqnarray}
that is, $\beta_i=0$, $i=2,8,9,11,12,15$.
The results of analysis are reported in Table~\ref{tab:real:mse}. We expect that the RBRIDGE estimators based on the prior information perform well compared to the other penalty estimators in terms of the measure of ${\rm MSE_{\beta_0}}$. The RBRIDGE$^1$ has an improvement $3.375$ times more than the BRIDGE and is at least $2.75$ times better than the others since it is estimated by using the restrictions based on the true parameter $\bbeta_{01}$. By looking at $\rm MSE_{\beta_{02}}$, the RBRIDGE$^2$ has an improvement $1.759$ times more than the BRIDGE and is at least $1.67$ times better than the others since it is constructed based on $\bbeta_{02}$. On the other hand, it can be also seen that the  RBRIDGE$^1$ outshines the others since it has the lowest prediction error, which is shown at ${\rm MSE_y}$, even better than the RBRIDGE$^2$. In the last column of Table~\ref{tab:real:mse} we also report the median of the number selected variables throughout replications.

\begin{table}[H]
\centering
\caption{Performance analysis of the estimators for the Air Pollution and Mortality Data}
\label{tab:real:mse}
\begin{tabular}{lrrrrrrr}
  \toprule
 & $\rm MSE_y$ & $\rm RMSE_y$ & $\rm MSE_{\beta_{01}}$ & $\rm RMSE_{\beta_{01}}$ & $\rm MSE_{\beta_{02}}$ & $\rm RMSE_{\beta_{02}}$ & \# variable  \\ 
  \midrule
LASSO & 46745.635 & 1.213 & 1038.526 & 2.756 & 16503.929 & 1.747 & 8 \\ 
  RIDGE & 47758.527 & 1.239 & 1414.969 & 3.755 & 15814.735 & 1.674 & 15 \\ 
  E-NET & 46344.879 & 1.203 & 1118.269 & 2.968 & 16499.733 & 1.746 & 9 \\ 
  SCAD & 53071.400 & 1.377 & 1169.844 & 3.104 & 16595.772 & 1.756 & 6 \\ 
  BRIDGE & 49280.423 & 1.279 & 1271.908 & 3.375 & 16621.214 & 1.759 & 7 \\ 
  RBRIDGE$^1$ & 38534.576 & \bf 1.000 & 376.837 & \bf 1.000 & 15846.628 & 1.677 & 7 \\ 
  RBRIDGE$^2$ & 46518.768 & 1.207 & 4447.006 & 11.801 & 9448.353 & \bf 1.000 & 9 \\ 
   \bottomrule
\end{tabular}
\begin{flushleft}                           
\footnotesize{$^1$the restrictions based on Stepwise Forward, $^2$the restrictions based on Stepwise Backward}
\end{flushleft}
\end{table}

In Table~\ref{tab:est}, we report the estimates along with the tuning parameters used in the proposed penalty estimators to better understanding of methods.
\begin{table}[H]
\centering
\caption{Estimated coefficients of Air Pollution and Mortality Data}
\label{tab:est}
\begin{adjustbox}{width=1\textwidth}
\begin{tabular}{llrrrrrrr}
  \toprule
 && LASSO & RIDGE & E-NET & SCAD & BRIDGE & RBRIDGE$^1$ & RBRIDGE$^2$ \\ 
 && $\lambda=2.023$
 & $\lambda=8.358$
 & $\lambda=2.168$ 
  & $\lambda=2.272$ 
 & $\lambda=373.815$ 
 & $\lambda=0.318$ 
  & $\lambda=15.883$  \\ 
 &&&& $\alpha= 0.850$ & & $q=0.800$ &  $q=0.050$ & $q=1.250$ \\
  \midrule
$\hbeta_1$& Precip & 15.037 & 16.168 & 15.198 & 16.869 & 15.805 & 16.460 & 16.498 \\ 
$\hbeta_2$&  Humidity & -- & 1.894 & 0.047 & -- & -- & -- & -- \\ 
$\hbeta_3$&  JanTemp & -12.101 & -12.028 & -12.077 & -18.120 & -13.551 & -19.256 & -22.965 \\ 
$\hbeta_4$&  JulyTemp & -6.000 & -6.815 & -6.257 & -11.029 & -7.013 & -10.956 & -13.347 \\ 
$\hbeta_5$&  Over65 & -- & -4.123 & -- & -- & -- & -- & -9.866 \\ 
$\hbeta_6$&  House & -- & -2.190 & -- & -6.036 & -- & -8.387 & -14.834 \\ 
$\hbeta_7$&  Educ & -8.786 & -7.303 & -8.498 & -13.269 & -10.156 & -14.338 & -19.485 \\ 
$\hbeta_8$&  Sound & -2.658 & -5.964 & -3.010 & -- & -- & -- & -- \\ 
$\hbeta_9$&  Density & 5.281 & 7.527 & 5.510 & 1.606 & 4.626 & -- & -- \\ 
$\hbeta_{10}$&  NonWhite & 35.548 & 29.234 & 35.519 & 45.147 & 38.156 & 46.530 & 44.580 \\ 
$\hbeta_{11}$&  WhiteCol & -0.066 & -2.243 & -0.281 & -- & -- & -- & -- \\ 
$\hbeta_{12}$&  Poor & -- & 2.203 & -- & -- & -- & -- & -- \\ 
$\hbeta_{13}$&  HC & -- & -3.583 & -- & -- & -- & -- & -47.870 \\ 
$\hbeta_{14}$&  NOX & -- & 3.276 & -- & -- & -- & -- & 50.958 \\ 
$\hbeta_{15}$&  SO2 & 14.471 & 15.255 & 14.574 & 14.231 & 14.090 & 14.284 & -- \\
   \bottomrule
\end{tabular}
\end{adjustbox}
\begin{flushleft}                           
\footnotesize{$^1$the restrictions based on Stepwise Forward, $^2$the restrictions based on Stepwise Backward}
\end{flushleft}
\end{table}

Finally, the 3D plot of the cross validation errors (CVE) of the RBRIDGE estimator versus $q$ and $log(\lambda)$ is plotted.
\begin{figure}[H]
\centering
\includegraphics[width=16cm,height=8cm]{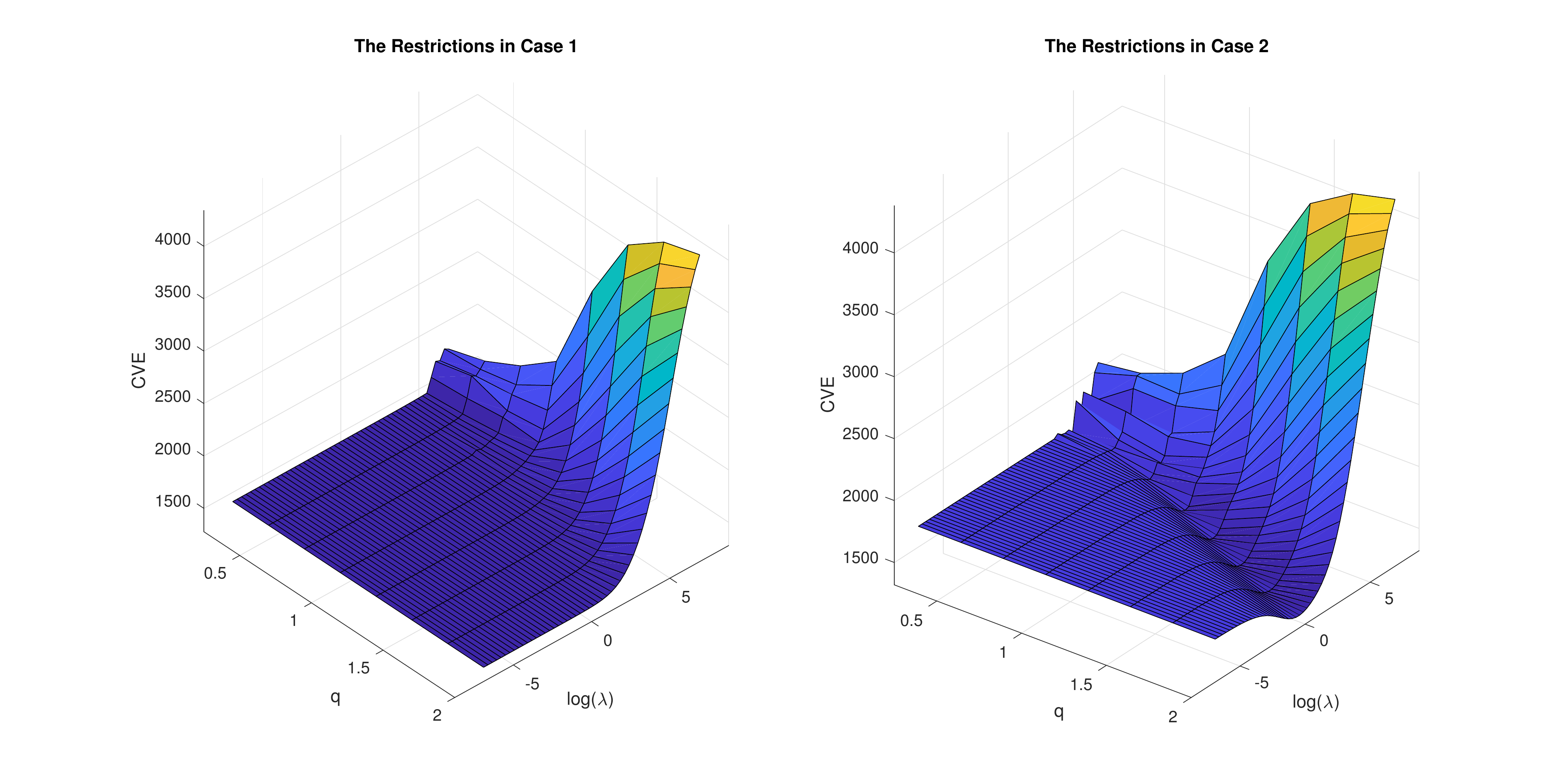}
\caption{The 3D plot of the CVE of the RBRIDGE estimator versus $q$ and $log(\lambda)$ for Air Pollution and Mortality Data  under the restrictions Stepwise Forward and Backward, respectively.
 \label{air_d:3d}}
\end{figure}

\subsection{Gorman--Toman Data}
A Ten-Factor data set first described by \cite{gorman1966selection} and used by several authors. c.f.,  \cite{hocking1967selection,hoerl1970ridge,gunst1976comparison,ozkale2014relative}. One may freely obtain this data from the {\it ridge} package in R, \cite{cule2019ridge}. The data set has $36$ observations, which shows one day of operation of a petroleum refining unit, on $10$ independent variables and one dependent variable. We consider three different scenarios on the restrictions following the explanations of \cite{ozkale2014relative}. 
\begin{itemize}[align=left]
\item[Case 1] \cite{gunst1976comparison} conducted that there exists multicollinearity among the variables of $X_1.X_5$ and $X_6$. \cite{ozkale2014relative} identified the restriction $\beta_1+\beta_5+\beta_6 = 0$. This restriction can be expressed as
$\bR_1 = \begin{bmatrix} 
 1 & 0 & 0 & 0 & 1 & 1 & 0 & 0 & 0 & 0
  \end{bmatrix}$
and $\br_1 = 0$.

\item[Case 2] Based on the $\rm C_p$ statistic, \cite{hoerl1970ridge} demonstrated that the first, fourth, ninth and tenth explanatory variables may be ignored since they are not significantly important. Hence, we have the restriction $\beta_1 = 0$, $\beta_4 = 0$, $\beta_9 = 0$ and $\beta_{10} = 0$ which yields
$\bR_2 = \begin{bmatrix}
\begin{tabular}{rrrrrrrrrr}
1 & 0 & 0 & 0 & 0 & 0 & 0 & 0 & 0 & 0 \\ 
  0 & 0 & 0 & 1 & 0 & 0 & 0 & 0 & 0 & 0 \\ 
  0 & 0 & 0 & 0 & 0 & 0 & 0 & 0 & 1 & 0 \\ 
  0 & 0 & 0 & 0 & 0 & 0 & 0 & 0 & 0 & 1 \\ 
\end{tabular}
  \end{bmatrix}$
  and $ \br_2 = \begin{bmatrix}
  0 \\
  0 \\
  0 \\
  0 \\
  \end{bmatrix}$.

\item[Case 3] Finally, following \cite{hocking1967selection}, the elements of the restriction are given by \linebreak $\bR_3 = \begin{bmatrix}
\begin{tabular}{rrrrrrrrrr}
0 & 0 & 0 & 0 & 1 & 0 & 0 & 0 & 0 & 0 \\ 
  0 & 0 & 0 & 0 & 0 & 1 & 0 & 0 & 0 & 0 \\ 
  0 & 0 & 0 & 0 & 0 & 0 & 1 & 0 & 0 & 0 \\ 
  0 & 0 & 0 & 0 & 0 & 0 & 0 & 1 & 0 & 0 \\ 
  0 & 0 & 0 & 0 & 0 & 0 & 0 & 0 & 1 & 0 \\ 
  0 & 0 & 0 & 0 & 0 & 0 & 0 & 0 & 0 & 1 \\ 
\end{tabular}
  \end{bmatrix}$
  and $ \br_3 = \begin{bmatrix}
  0 \\
  0 \\
  0 \\
  0 \\
  0 \\
  0 \\
  \end{bmatrix}$, that is, the first four variable are significantly important and the rests are considered as nuisance parameters meaning  $\beta_i=0, i=5,\dots,10$.
\end{itemize}

To evaluate the performance of the listed estimators, we only use the prediction error defined by \eqref{mse_y}. We report the results in Table~\ref{gorman_mse}. It can be seen that the RBRIDGE$^2$ outperforms the others. Also, the RBRIDGE estimators based on the restrictions in cases 1 \& 2 are superior compared to the BRIDGE estimator. Again, the last row in Table~\ref{gorman_mse}, we report the median of the number of selected variables throughout replications. 
\begin{table}[H]
\caption{Performance analysis of the estimators for the Gorman--Toman Data}
\label{gorman_mse}
\begin{adjustbox}{width=1\textwidth}
\centering
\begin{tabular}{rrrrrrrrr}
  \toprule
 & LASSO & RIDGE & E-NET & SCAD & BRIDGE & RBRIDGE$^1$ & RBRIDGE$^2$ & RBRIDGE$^3$ \\
 \midrule
    $\rm MSE_y$ & 0.45 & 0.38 & 0.44 & 0.56 & 0.44 & 0.40 & 0.34 & 0.54 \\ 
$\rm RMSE_y$   & 1.31 & 1.12 & 1.30 & 1.63 & 1.28 & 1.17 & \bf 1.00 & 1.58 \\ 
 \# variable & 9.00 & 10.00 & 9.00 & 7.00 & 10.00 & 10.00 & 6.00 & 3.00 \\ 
   \bottomrule
\end{tabular}
\end{adjustbox}
\begin{flushleft}                           
\footnotesize{$^1$the restriction in case 1, $^2$the restriction in case 2,
$^3$the restriction in case 3}
\end{flushleft}
\end{table}

%% EX 3
\subsection{Lu2004 Gene Data}
\label{ssec:lu2004}
This data comes from a gene-expression study investigating the relation
of aging and gene expression in the human frontal cortex \cite{lu2004gene}, and it is available at \href{https://www.ncbi.nlm.nih.gov/geo/query/acc.cgi?acc= GSE1572}{https://www.ncbi.nlm.nih.gov/geo/query/acc.cgi?acc= GSE1572}. In the raw data, there are $n=30$ patients whose age are from $26$ to $106$ years, and the expression of $p = 12,625$ genes was measured by microarray technology. We use the data following prescreening and preprocessing of \cite{zuber2011high}, the selected 403 genes and response variable may freely be obtained from the {\it care} package in R, \cite{zuber2014care}. Here we do not have a piece of prior information regarding this data, and one may use a penalty estimation method to identify important variables. To this end, we consider three cases for restrictions:

\begin{itemize}[align=left]
\item[Case 1] We first apply the LASSO. It identifies $\bbeta_1$ as significantly important coefficients, while $\bbeta_2$ is the non-important coefficients vector which those are not expected contribute to the estimating of the response. Hence, we consider $\bR=[\bzero,\bI]$, where $\bzero$ and $\bI$ are suitable sizes zero and identity matrices, respectively, such that $\bbeta_2 = \br = {\bf 1}_{nz}^{\top}\cdot 0$, where $nz$ is the number of zeros for each method.
\item[Case 2] Just like Case 1, except in this case we apply the SCAD as the variable selection method.
\item[Case 3] Just like Case 1, except the E-NET is applied as the variable selection method.
\end{itemize}

\begin{table}[H]
\caption{Performance analysis of the  estimators for the Lu2004 Gene Data
\label{lu2004gene_mse}}
\centering
\begin{tabular}{lrrrr}
  \toprule
 & $\rm MSE_y$ & $\rm RMSE_y$ & $\rm RMSE^*_y$ & \# variables \\ 
  \midrule
LASSO & 3613.476 & 11.694 & -- & 15 \tikzmark[xshift=0.1em]{a}\\ 
  RIDGE & 2316.688 & 7.497 & -- & 403 \tikzmark[xshift=0.1em]{b}\\ 
  E-NET & 3180.204 & 10.292 & -- & 27 \tikzmark[xshift=0.1em]{c}\\ 
  SCAD & 4626.370 & 14.972 & -- & 7 \tikzmark[xshift=0.1em]{d}\\ 
  BRIDGE & 2981.028 & 9.647 & -- & 388 \tikzmark[xshift=0.1em]{e}\\ 
  RBRIDGE$^1$ & 372.137 & 1.204 & 9.710 & 26 \tikzmark[xshift=0.1em]{f}\\ 
  RBRIDGE$^2$ & 1612.334 & 5.218 & 2.869 & 9 \tikzmark[xshift=0.1em]{g}\\ 
  RBRIDGE$^3$ & 308.995 & \bf 1.000 & 10.292 & 27 \tikzmark[xshift=0.1em]{h}\\ 
   \bottomrule
\end{tabular}
\drawcurvedarrow[bend left=90,-stealth]{a}{f}
\drawcurvedarrow[bend left=90,-stealth]{d}{g}
\drawcurvedarrow[bend left=90,-stealth]{c}{h}
\begin{flushleft}                           
\footnotesize{$^1$the restriction in case 1, $^2$the restriction in case 2,
$^3$the restriction in case 3. $\rm RMSE^*_y$ stands for the relative performance of the LASSO vs RBRIDGE$^1$, SCAD vs RBRIDGE$^2$ and E-NET vs RBRIDGE$^3$, respectively. As an instance, the RBRIDGE reduced the prediction error $9.71$ times if the restriction is selected by LASSO.}
\end{flushleft}
\end{table}

As it can be seen from Table~\ref{lu2004gene_mse}, the RBRIDGE$^3$ outperforms the others, that is, it improves the performance in prediction error sense when it uses the prior information provided by the E-NET for the restriction. We also note that the RBRIDGE$^3$ has an improvement $10.292$ times compared to the E-NET itself; see the column of $\rm RMSE^*_y$. If a piece of prior information is used by LASSO or SCAD, our suggested method has an impressive improvement. The last column of the Table~\ref{lu2004gene_mse} shows the median values of the selected important variables after 100 replications.

\subsection{Eye Data}
This data is extracted from the study of \cite{scheetz2006regulation}, and it is originally available at \href{https://www.ncbi.nlm.nih.gov/geo/query/acc.cgi?acc= GSE5680}{https://www.ncbi.nlm.nih.gov/geo/query/acc.cgi?acc= GSE5680}. In the raw data, there are $n=120$ laboratory rats (Rattus norvegicus) to gain a broad perspective of gene regulation in the mammalian eye and to identify genetic variation relevant to human eye disease. There are over $31,000$ gene probes represented on an Affymetrix expression microarray. Following \cite{li2015flare}, we use $200$ gene probes in order to estimate the expression of the TRIM32 gene as a response. This data may freely be obtained from the {\it flare} package; see \cite{li2015flare}. We follow exactly the same structure as in Cases 1 -- 3 in Section~\ref{ssec:lu2004} to formulate restrictions as prior information. In Table~\ref{eye_mse}, we report the analysis results. For this example, it can be understood that all RBRIDGE estimators outperform the penalty counterparts, and the RBRIDGE$^1$, which is estimated by using the preliminary information obtained from the LASSO has the best performance among all.

\begin{table}[H]
\caption{Perforamance analysis of the estimators for the Eye Data
\label{eye_mse}}
\centering
\begin{tabular}{lrrrr}
  \toprule
 & $\rm MSE_y$ & $\rm RMSE_y$ & $\rm RMSE^*_y$ & \# variables \\ 
  \midrule
LASSO & 0.393 & 1.531 & -- & 19 \tikzmark[xshift=0.1em]{a}\\ 
  RIDGE & 0.362 & 1.411 & -- & 200 \tikzmark[xshift=0.1em]{b} \\ 
  E-NET & 0.385 & 1.499 & -- & 28 \tikzmark[xshift=0.1em]{c} \\ 
  SCAD & 0.478 & 1.859 & -- & 9 \tikzmark[xshift=0.1em]{d} \\ 
  BRIDGE & 0.417 & 1.625 & -- & 200 \tikzmark[xshift=0.1em]{e} \\ 
  RBRIDGE$^1$ & 0.257 & \bf 1.000 & 1.531 & 24 \tikzmark[xshift=0.1em]{f} \\ 
  RBRIDGE$^2$ & 0.278 & 1.083 & 1.716 & 10 \tikzmark[xshift=0.1em]{g} \\ 
  RBRIDGE$^3$ & 0.282 & 1.099 & 1.363 & 33 \tikzmark[xshift=0.1em]{h} \\ 
   \bottomrule
\end{tabular}
\drawcurvedarrow[bend left=90,-stealth]{a}{f}
\drawcurvedarrow[bend left=90,-stealth]{d}{g}
\drawcurvedarrow[bend left=90,-stealth]{c}{h}
\begin{flushleft}                           
\footnotesize{$^1$the restriction in case 1, $^2$the restriction in case 2,
$^3$the restriction in case 3. $\rm RMSE^*_y$ stands for the relative performance of the LASSO vs RBRIDGE$^1$, SCAD vs RBRIDGE$^2$ and E-NET vs RBRIDGE$^3$, respectively.}
\end{flushleft}
\end{table}

%%%%%%%%%%%%%%%%%%%%%%%%%%%%%%%%%%%%%%%%%%%%%%%%%%%%%%%%%%%%
%%%%%%%%%%%%%%%%%%Conclusions%%%%%%%%%%%%%%%%%%%%%%%%%%%%%%%
%%%%%%%%%%%%%%%%%%%%%%%%%%%%%%%%%%%%%%%%%%%%%%%%%%%%%%%%%%%%
\section{Conclusions}
\label{sec:conc}
We used the local quadratic approximation  (LQA) tio obtain a closed-form restricted BRIDGE (RBRIDGE) estimator. 
We studied the low dimensional properties of the proposed estimator and compared its performance numerically with some well-known penalty estimators. Using an extensive simulation study and by analyzing four real data sets we demonstrated the superiority of the proposed RBRIDGE estimator in the sense of better model accuracy and variable selection, under restriction.
One interesting result is that the number of important co-variates between the restriction matrix and the estimation of the BRIDGE estimator may differ since the $q$-norm penalty may select variables. In this case, the results show that the RBRIDGE estimator has better performance, according to the given measures.
Overall, the observations from numerical studies suggest that the proposed RBRIDGEs perform well in estimation accuracy and model selection when the are some linear restrictions present in the study. 

For $0<q\leq1$, \cite{hunter2005variable} showed that the LQA is a special case of a minorization-maximization (MM) algorithm and guarantees the ascent property of maximization problems; and proposed a perturbed LQA. On the other hand, the local linear approximation (LLA) of \cite{zou2008one} enjoys three significant advantages over the local quadratic approximation  (LQA) and the perturbed LQA. For further research, our results can be further investigated using the MM and LLA methods.

\section*{Appendix}
Here we sketch the proofs of theoretical results.
%\noindent\textbf{Proof of Theorem \ref{theorem:restricted bridge}}:
\begin{proof}[Proof of Theorem \ref{theorem:restricted bridge}]
Assume a local point $\bbeta^o$. Using the LQA of Fan and Li, the purpose is to solve the following optimization problem
\begin{eqnarray*}
&&\min_{\bbeta\in\mathbb R^p} (\bY-\bX_n\bm\beta)^\top(\bY-\bX_n\bm\beta)\cr
\textnormal{s.t.}&& \bbeta^\top\diag\left( \left| \hat\beta_1^{o} \right|^{q-2},\dots,\left| \hat\beta_p^{o} \right|^{q-2}  \right)\bbeta^\top\leq t,\;t>0\quad\mbox{and} \quad\bR\bm\beta=\br
\end{eqnarray*}
Using dual representation, we minimize the following objective function
\begin{equation*}
    S_{\mathrm{Pen}}(\bm\beta)=(\bY-\bX_n\bm\beta)^\top(\bY-\bX_n\bm\beta)+\bbeta^\top\bSigma_{\lambda}(\hbbeta^o)\bbeta^\top+\bm\gamma^\top(\bR\bm\beta-\br),
\end{equation*}
where $\bm\gamma=(\gamma_1,\ldots,\gamma_q)^\top$ is the Lagrangian vector of multipliers. Extending the terms and differentiating w.r.t $\bm\beta$, after some modifications and using \eqref{eq-2.3}, we can get
\begin{eqnarray}\label{eq-6.1}
    \bm\beta&=&\left(\bX^\top\bX+\bSigma_{\lambda}(\hbbeta^o)\right)^{-1}\left(\bX_n^\top\bY+\bR^\top\bm\gamma\right)\cr
    &=&\hbbeta_n+\left(\bX^\top\bX+\bSigma_{\lambda}(\hbbeta^o)\right)^{-1}\bR^\top\bm\gamma.
\end{eqnarray}
Applying $\bR\bm\beta=\br$ to the RHS of \eqref{eq-6.1}, we have
\begin{equation*}
    \bR\hbbeta_n+\bR\left(\bX^\top\bX+\bSigma_{\lambda}(\hbbeta^o)\right)^{-1}\bR^\top\bm\gamma=\br
\end{equation*}
that we can obtain
\begin{equation*}
    \bm\gamma=\left[\bR\left(\bX^\top\bX+\bSigma_{\lambda}(\hbbeta^o)\right)^{-1}\bR^\top\right]^{-1}\left(\br-\bR\hbbeta_n\right).
\end{equation*}
Substituting $\bm\gamma$ in \eqref{eq-6.1} gives the required result. 
\end{proof}
%\noindent\textbf{Proof of Theorem \ref{theorem:MSE}}:
\begin{proof}[Proof of Theorem \ref{theorem:MSE}]
We may write
\begin{eqnarray*}
\hat\bbeta_n^{\mathrm{R}}&=&\bM\bS\hat\bbeta_n+\bS^{-1}\bR^\top(\bR\bS^{-1}\bR^\top)^{-1}\br\cr
&=&\bM\bS\hat\bbeta_n+\bS^{-1}\bR^\top(\bR\bS^{-1}\bR^\top)^{-1}\bR\bbeta\cr
&=&\bM\bS\hat\bbeta_n-\bM\bS\bbeta+\bbeta
\end{eqnarray*}
Hence
\begin{eqnarray*}
E(\hat\bbeta_n^R)&=&\bM\bS\bS^{-1}\bC_n\bbeta-\bM\bS\bbeta+\bbeta\cr
&=&-\bM\bSigma_\lambda(\hbbeta^o)\bbeta+\bbeta.
\end{eqnarray*}
On the other hand, simple algebra yields $\mathrm{Cov}(\hat\bbeta_n^R)=\bM\bC_n\bM=\bM\bS\bM$. Hence, we get
\begin{eqnarray*}
\mathrm{MSE}(\hat\bbeta_n^R)&=&\mathrm{tr}\mathrm{Cov}(\hat\bbeta_n^R)+\mathrm{Bias}(\hat\bbeta_n^R)^\top\mathrm{Bias}(\hat\bbeta_n^R)\cr
&=&\mathrm{tr}(\bM\bS\bM)+\mathrm{tr}(\bM\bSigma_\lambda(\hbbeta^o)\bbeta^\top\bbeta\bSigma_\lambda(\hbbeta^o)\bM)\cr
&=&\mathrm{tr}(\bM[\bS+\bSigma_\lambda(\hbbeta^o)\bbeta^\top\bbeta\bSigma_\lambda(\hbbeta^o)]\bM)
\end{eqnarray*}
\end{proof}
%\noindent\textbf{Proof of Theorem \ref{theorem:Consistency}}:
\begin{proof}[Proof of Theorem \ref{theorem:Consistency}]
Let
\begin{equation*}
\dot{\bbeta}_n=\argmin_{\bbeta}\left\{\frac1n\sum_{i=1}^{n}(Y_i-\bm x^\top\bm\beta)^2+\frac{\lambda_n}{n}\sum_{j=1}^{p}|\beta_j|^q\right\}
\end{equation*}
Using Theorem 1 of \cite{knight2000asymptotics}, if $\lambda_n=o(n)$, then $\dot{\bbeta}_n\overset{\mathcal{P}}{\to}\bbeta$. Under the regularity condition (ii), $\bS_n\to\bC$. Hence, we get $\hat\bbeta_n^{\mathrm{R}}\overset{\mathcal{P}}{\to}\bbeta-\bC^{-1}\bR^\top(\bR\bC^{-1}\bR^\top)^{-1}(\bR\bbeta-\br)$. Then, the desired result follows unsder the sub-space restriction $\bR\bbeta=\br$.
\end{proof}

%%%%%%%%%%%%%%%%%%%%%%%%%%%%%%%%%%%%%%%%%%
%%%%%%%%%%%%%%%%% REFs %%%%%%%%%%%%%%%%%%
%%%%%%%%%%%%%%%%%%%%%%%%%%%%%%%%%%%%%%%%%%
\bibliography{refs}
\bibliographystyle{apalike}

\end{document}